\documentclass{dcds}
\usepackage{amsfonts}
\usepackage{amsmath}
  \usepackage{paralist}
  \usepackage{graphics} %% add this and next lines if pictures should be in esp format
  \usepackage{epsfig} %For pictures: screened artwork should be set up with an 85 or 100 line screen
 \usepackage[colorlinks=true]{hyperref}
\hypersetup{urlcolor=blue, citecolor=red}

\def\currentvolume{33}
 \def\currentissue{7}
  \def\currentyear{2013}
   \def\currentmonth{July}
    \def\ppages{X--XX}
     \def\DOI{10.3934/dcds.2013.33.xx}

\newtheorem{thm}{Theorem}[section]

\newtheorem{lema}{Lemma}[section]
\newtheorem{prop}{Proposition}[section]

\theoremstyle{definition}
\newtheorem{dif}{Definition}[section]

\renewcommand{\theequation}{\thesection.\arabic{equation}}

\title[Global Weak Solutions to a General Liquid Crystals System]
      {Global Weak Solutions to a General Liquid Crystals System}

\author[Yuming Chu, Yihang Hao and Xiangao Liu]{}

\subjclass{76N10, 35Q35, 35Q30.}
 \keywords{Bulk free energy, finite energy weak solutions, Navier-Stokes equations.}
\thanks{This work was supported partly by NSFC grant 11071043, 11131005, 11071069.}
\thanks{* Corresponding author.}
\email[Y.M Chu] {chuyuming@hutc.zj.cn}
 \email[Y.H Hao]{10110180022@fudan.edu.cn}
 \email[X.G Liu]{xgliu@fudan.edu.cn}

\begin{document}
\maketitle

\centerline{\scshape Yuming Chu }
\medskip
{\footnotesize
% please put the address of the first author
 \centerline{Department of Mathematics, Huzhou Teachers College}
\centerline{Zhejiang Huzhou, China}}
% Enter the first author's name and address:
\medskip

\centerline{\scshape Yihang Hao* and Xiangao Liu }
\medskip
{\footnotesize
% please put the address of the first author
 \centerline{School of Mathematical Sciences, Fudan University}
   \centerline{Shanghai, China}
   } % Do not forget to end the {\footnotesize by the sign }

%\medskip
%
%\centerline{\scshape Xiangao Liu}
%\medskip
%{\footnotesize
% % please put the address of the second  and third author
% \centerline{School of Mathematical Sciences, Fudan University}
%   \centerline{Shanghai, China}
%   }

\bigskip

% The name of the associate editor will be entered by an editorial staff
% "Communicated by the associate editor name" is not needed for special issue.
 \centerline{(Communicated by Fanghua Lin)}

%The abstract of your paper
\begin{abstract}
We prove the global existence of finite energy weak solutions to the
general liquid crystals system. The problem is studied in bounded
domain of $\mathbb{R}^3$ with
 Dirichlet boundary conditions and the whole space $\mathbb{R}^3$.
\end{abstract}

%The title of your section 1
\section {Introduction}

\par Liquid crystals were discovered in 1888 by F.Reinitzer and O.Le-hmann. They are
often viewed as intermediate states between the solids and fluids, whose
molecular arrangements give rise to preferred directions. As a result, they retain several
different features: mechanical, electrical, magnetic and optical properties. According to molecular
arrangements, it has been widely supported that liquid crystals can be classified into three types: nematics,
cholesterics and smectics. The historical development of liquid crystals confronts two theories, including the swarm theory and the distortion theory. The former theory is well established but only applied to nematics and cholesterics, while the later one, successfully explaining the interactions of nematics with magnetic fields, is not so well-known for us. Based on the predecessors' work, in the 1960's, Ericksen and Leslie
established the Hydrodynamic theory of nematic liquid crystals system(see \cite{ericksen1}, \cite{ericksen2},
\cite{ericksen3}, \cite{gennes}, \cite{stephen} and \cite{xie}):
\begin{align}
&\rho_t+\mathrm{div}(\rho u)=0\label{fequa1},\\
&(\rho u)_t+\mathrm{div}(\rho u\otimes u)=\rho\overline{F}+\mathrm{div}\sigma,\label{fequa2}\\
&\rho_1\frac{dw}{dt}=\rho_1\overline{G}+g+\mathrm{div}\pi,\label{fequa3}
\end{align}
where $\rho\geq0$, $u=(u_1,u_2,u_3)$, $d=(d_1,d_2,d_3)$ are the fluid density, velocity and molecular direction respectively. $\overline{F}$ denotes the external body force, $\overline{G}$ the external body force for the direction movement,
 $g$ the internal body force for the direction movement and $\rho_1\frac{dw}{dt}$ the angular movement per unit time. In low frequency, $\rho_1\frac{dw}{dt}$ is so small that can be neglected. The equations (\ref{fequa1})-(\ref{fequa3}) represent the conservation of mass,
  linear momentum, and angular momentum respectively.
% and $I$ is the dynamic moment of inertia per unit volume.
 $\sigma,\ g$ and $\pi$ satisfy the following constitutive relations:
\begin{align}
\sigma_{ij}&=\left(-p-\rho^2 \frac{\partial H}{\partial\rho}\right)\delta_{ij}-\frac{\partial(\rho H)}{\partial d_{k,i}}d_{k,j}+\widehat{\sigma}_{ij},\label{fcon1}\\
\pi_{ij}&=\frac{\partial(\rho H)}{\partial d_{i,j}},\label{fcon2}\\
g_i&=-\frac{\partial(\rho H)}{\partial d_i}+\kappa d_i+\widehat{g}_i.\label{fcon3}
\end{align}
where $p=a\rho^\gamma$ denotes the pressure, $\rho H$ the bulk free energy, and let $d_{i,j}$ represent $\frac{\partial d_i}{\partial x_j}$. In equation (\ref{fcon3}), $\kappa$ is Lagrange multiplier constraint to $|d|=1$. According to Frank's formula \cite{virga} (Chapter 3), we can get the bulk free energy of nematic types.
\begin{align}
2\rho H=k_1&(\mathrm{div} d)^2+k_2(d\cdot \mathrm{curl}d)^2+k_3|d\wedge \mathrm{curl}d|^2\nonumber\\
+&(k_2+k_4)\left(tr(\nabla d)^2-(\mathrm{div}d)^2\right).\label{fcon7}
\end{align}
Likewise, from \cite{xie} we have
\begin{align}
\widehat{\sigma}&=\mu_1\left(d^TAd\right)d\otimes d+\mu_2N\otimes d+\mu_3d\otimes N+\mu_4A\nonumber\\
&+\mu_5(Ad)\otimes d+\mu_6d\otimes(Ad)+\mu_7tr(A)I, \label{fcon4}\\
\widehat{g}&=\lambda_1N+\lambda_2Ad. \label{fcon5}
\end{align}
Here $d\otimes d$ denotes a matrix whose $(i,j)$-th entry is given by $d_id_j$. We use the following notations:
\begin{align}
&w=\frac{\partial d}{\partial t}+(u\cdot\nabla)d,\nonumber\\
&N=w-\Omega d,\quad A=\frac{1}{2}\left(\nabla u+(\nabla u)^T\right),\nonumber\\
&\Omega=\frac{1}{2}(\nabla u-(\nabla u)^T),\nonumber
\end{align}
here $N=(N_1,N_2,N_3)$ is used to describe the director movements in satellited coordinates, and $w=(w_1,w_2,w_3)$
 is the material derivatives of $d$. And $\lambda_i,\ \mu_i$ satisfy the following formulas:
\begin{align}
&\lambda_1=\mu_2-\mu_3<0,\label{coef1}\\
&\lambda_2=\mu_5-\mu_6=-(\mu_2+\mu_3),\label{coef2}\\
&\mu_5+\mu_6\geq0,\quad \mu_1\geq0,\label{coef3}\\
&\mu_4\geq0,\quad\mu_7\geq0.\label{coef4}
\end{align}
(\ref{coef2}) is called Parodi's condition, which is derived from the onsager\ reciprocal\ relation, see\cite{parodi},
\cite{gennes}, \cite{xie}.
For simplicity, we assume $k_4=0$, $k_1=k_2=k_3=1$. Observing that
\begin{align}
&|\nabla d|^2=\mathrm{tr}(\nabla d)^2+(\mathrm{curl}d)^2,\nonumber\\
(\mathrm{curl}&d)^2=(d\cdot \mathrm{curl}d)^2+|d\wedge \mathrm{curl}d|^2,\nonumber
\end{align}
we have the bulk free energy
\begin{align}
2\rho H=|\nabla d|^2.\nonumber
\end{align}
In statics and low frequency, if we set $\overline{G}=0$ and neglect $\rho_1\frac{dw}{dt}$, the equation (\ref{fequa3}) reads
\begin{align}
\Delta d+\kappa d=0.\label{minenergy}
\end{align}
Indeed equation (\ref{minenergy}) is the Euler equation of  the minimum bulk free energy $\rho H$ under the restriction $|d|=1$. As a consequence, we have $\kappa=|\nabla d|^2$. Because the nonlinear term $|\nabla d|^2d$ is bad for getting weak solution, we use the Ginzberg-Landau approximate function
\begin{align}
2\rho H=|\nabla
d|^2+2F(d),\label{bulkfe}
\end{align} where $F(d)=\frac{1}{4\varepsilon^2}(|d|^2-1)^2$ for fixed $ \varepsilon>0$. Thus we use
\begin{eqnarray}
-\Delta d+\frac{1}{\varepsilon^2}(|d|^2-1) d=0\label{appminenergy}
\end{eqnarray}
instead of (\ref{minenergy}).

In this paper, we set $\rho_1\frac{d w}{dt}=0, \overline{F}=G=0, $ $\mu_1=\mu_5=\mu_6=\lambda_2=0$. Using (\ref{coef1})-(\ref{coef4}), we have $$\mu_2=-\mu_3<0,\ \lambda_1=-2\mu_3<0,\ \mu_4\geq0,\ \mu_7\geq0.$$  Finally, the liquid crystals system becomes
\begin{align}
&\rho_t+\mathrm{div}(\rho u)=0\label{equa1},\\
&(\rho u)_t+\mathrm{div}(\rho u\otimes u)+\nabla\left(p-\frac{1}{2}|\nabla d|^2-F(d)\right)+\nabla\cdot(\nabla d\odot\nabla d)=\mathrm{div}\widehat{\sigma},\label{equa2}\\
&\lambda_1d_t+\lambda_1u\cdot\nabla d-\lambda_1\Omega d+\Delta d-f(d)=0,\label{equa3}
\end{align}
where $\widehat{\sigma}$ is given by
\begin{eqnarray}
\widehat{\sigma}=\mu_3(N\otimes d-d\otimes N)+\mu_4A+\mu_7tr(A)I.
\end{eqnarray}
After a simple computation, we find $|d|^2\leq1$ by the maximum principle.
\vskip0.5cm
\noindent{\bf The case of bounded domain $D$:}\\

We are interested in the global weak solutions to the system (\ref{equa1})-(\ref{equa3})  in a bounded domain $D\subset\mathbb{R}^3$
with  initial conditions:
\begin{eqnarray}
\left\{\begin{array}{ll}
&\rho(x,0)=\rho_0(x)\geq 0\quad a.e.\ in\ D,\ \rho_0\in L^\gamma(D), \\
&(\rho u)(x,0)=q(x),\ q(x)=0\quad a.e. \ on\ \{\rho_0(x)=0\},\ \frac{|q|^2}{\rho_0}\in L^1(D),\\
&d(x,0)=d_0(x), \quad |d_0(x)|=1,\ d_0\in H^2(D),
\end{array}\right.\label{ini1}
\end{eqnarray}
and the boundary conditions
\begin{eqnarray}
u(x,t)=0,\quad d(x,t)=d_0(x),\quad \quad(x,t)\in\partial D\times(0,\infty).\label{boun1}
\end{eqnarray}

In order to give the definition of the weak solutions, we firstly describe energy inequality
\begin{eqnarray}
\frac{d}{dt}E(t)\leq-\int_{D}\left[\mu_4|A|^2-\lambda_1|N|^2+\mu_7(divu)^2\right],\label{energy}
\end{eqnarray}
where
\begin{eqnarray}
E=\int_{D}\left[\frac{1}{2}\rho|u|^2+\frac{1}{2}|\nabla d|^2+\frac{1}{\gamma-1}p+F(d)\right].\nonumber\\
\end{eqnarray}
It reflects the energy dissipation property of the flow of liquid crystals.

Multiplying (\ref{equa1}) by $B'(\rho)$, we formally have
\begin{eqnarray}
(B(\rho))_t+div(B(\rho)u)+b(\rho)divu=0,\label{renorm1}
\end{eqnarray}
where $B$ is a smooth function and
\begin{eqnarray}
b(\rho)=B'(\rho)\rho-B(\rho).\label{renorm2}
\end{eqnarray}

\begin{dif}\label{defi}
We call $(\rho,u,d)$ is a finite energy weak solution of (\ref{equa1})-(\ref{equa3})
in bounded domain $D\subset\mathbb{R}^3$ with initial and boundary conditions (\ref{ini1}) and (\ref{boun1}),
if it satisfies the following conditions:
\begin{itemize}
\item $\rho\in L^\infty(0,T;L^\gamma(D))$, $\rho\geq0$ a.e. in $(0,T)\times D$; $d\in L^2(0,T;H^2(D))\cap L^\infty(0,T;H^1(D))\cap L^\infty((0,T)\times D)$, and $u\in L^2(0,T;H^1_0(D))$.
\item the energy $E(t)$ is locally integrable on $(0,T)$, and (\ref{energy}) holds in $D'(0,T)$;
\item $(\rho,u,d)$ satisfies (\ref{equa1})-(\ref{equa3}) in $D'((0,T)\times D)$, and (\ref{equa1}) holds
in $D'((0,T)\times \mathbb{R}^3)$, provided $(\rho,u)$ is prolonged to be zero on $\mathbb{R}^3\backslash D$;
\item equation (\ref{equa1}) is satisfied  in the sense of renormalized solutions, that is (\ref{renorm1})
holds in $D'((0,T)\times D)$ for any $B\in C[0,\infty)\cap C^1(0,\infty)$, $b\in C[0,\infty)$ bounded on $[0,\infty)$,
$B(0)=b(0)=0$.
\end{itemize}
\end{dif}

Then we have the following result:

\begin{thm}\label{th1}
Assume $D\subset \mathbb{R}^3$ is a bounded domain of the class $C^{2+\nu}$, $\nu>0$, and $\gamma>\frac{3}{2}$. Then the system (\ref{equa1})-(\ref{equa3}) with (\ref{ini1})-(\ref{boun1}) has a finite energy weak solution defined by Definition \ref{defi} for any $T<\infty$.
\end{thm}

\vskip0.5cm
\noindent{\bf The case of whole space $\mathbb{R}^3$:}\\

Let Banach space
\begin{eqnarray}
\mathcal{H}(\mathbb{R}^3)=\{\left.d\ \right|\ d\in L^\infty(\mathbb{R}^3),\ d\in \dot{H}^1(\mathbb{R}^3)\},\nonumber
\end{eqnarray}
and the Orlicz space $L_2^p(\mathbb{R}^3)$(see \cite{lions2} pp.288)
\begin{eqnarray}
L_2^p(\mathbb{R}^3)=\left\{f\ \in\ L_{loc}^1(\mathbb{R}^3),\ \left.f\right|_{|f|\leq\frac{1}{2}}\ \in\ L^2(\mathbb{R}^3),\ \left.f\right|_{|f|\geq\frac{1}{2}}\ \in\ L^p(\mathbb{R}^3)\right\}.\nonumber
\end{eqnarray}
We consider our problem (\ref{equa1})-(\ref{equa3}) in the whole space $\mathbb{R}^3$ with initial conditions:
\begin{eqnarray}
\left\{\begin{array}{ll}
&\rho(x,0)=\rho_0(x)\geq 0\quad a.e.\ in\ \mathbb{R}^3, \\
&(\rho u)(x,0)=q(x),\ q(x)=0\quad a.e. \ on\ \{\rho_0(x)=0\},\ \frac{|q|^2}{\rho_0}\in L^1(\mathbb{R}^3),\\
&d(x,0)=d_0(x), \quad |d_0(x)|=1,\ d_0\in\mathcal{H}(\mathbb{R}^3)\cap \dot{H}^2(\mathbb{R}^3),
\end{array}\right.\label{ini2}
\end{eqnarray}
where
\begin{eqnarray}
0\le\int_{\mathbb{R}^3}(\rho_0)^\gamma-\gamma(\rho_0-1)-1\leq C_0.\label{ini3}
\end{eqnarray}
One can obtain the energy inequality
\begin{align}
\frac{d}{dt}E(t)\leq -\int_{\mathbb{R}^3}\left[\mu_4|A|^2-\lambda_1|N|^2+\mu_7(divu)^2\right],\label{henergy}
\end{align}
where
\begin{eqnarray}
E=\int_{\mathbb{R}^3}\left[\frac{1}{2}\rho|u|^2+\frac{1}{2}|\nabla d|^2+\frac{\rho^\gamma-\gamma(\rho-1)-1}{\gamma-1}+F(d)\right].\nonumber
\end{eqnarray}
We define the weak solution in the sense of following

\begin{dif}\label{defi1}
We call $(\rho,u,d)$ is a weak-solution of (\ref{equa1})-(\ref{equa3}) in $\mathbb{R}^3$
if it satisfies the following conditions:
\begin{itemize}
\item $\rho-1\in L^\infty(0,T;L^\gamma(\mathbb{R}^3))$ if $\gamma\geq2$ and $\rho-1\in L^\infty(0,T;L^\gamma_2(\mathbb{R}^3))$ if $\gamma<2$, $\rho\geq0$ a.e. in $(0,T)\times \mathbb{R}^3$, $u\in L^2(0,T;H^1_0(\mathbb{R}^3))$, $d\in L^\infty(0,T;\mathcal{H}(\mathbb{R}^3))\cap L^2(0,T;\dot{H}^2(\mathbb{R}^3))$, $|d|\leq1$ a.e. in $(0,T)\times \mathbb{R}^3$;
\item the energy $E(t)$ is locally integrable on $(0,T)$, and (\ref{henergy}) holds in $D'(0,T)$;
\item $(\rho,u,d)$ satisfies (\ref{equa1})-(\ref{equa3}) in $D_{loc}'((0,T)\times \mathbb{R}^3)$.
\end{itemize}
\end{dif}
Then we have
\begin{thm}\label{th2}
Let $\gamma>\frac{3}{2}$. The system (\ref{equa1})-(\ref{equa3}) with (\ref{ini2})-(\ref{ini3}) has a weak-solution defined by Definition \ref{defi1} for any $T<\infty$.
\end{thm}

\par  There are a lot of results for the incompressible and compressible
liquid crystals systems. For the incompressible case with a constant density,
 F. Lin and C. Liu systematically studied the existence and partial regularity of
 weak solutions in their papers\cite{ll1}, \cite{ll2}, \cite{ll3}. For the density is not constant,
X. Liu and Z. Zhang \cite{lz}  proved existence of
weak solutions in a bounded domain of $\mathbb{R}^3$, also to see F. Jiang and Z.
Tan \cite{jt} for weakening assumption of \cite{lz}.  For the compressible case, in 2009, X. Liu and J.
Qing \cite{liu} firstly proved the existence of weak solutions to
liquid crystals system (more simpler  than one considered in present paper) in a bounded domain of $\mathbb{R}^3$. Similar results
also see D. Wang, Y. Cheng \cite{cw}. We also note that papers by S. Ding, C. Wang and H. Wen  \cite{dww1} for one dimension case and by
 X. Hu, D.Wang \cite{hw} for the Besov space case.

\par In this paper, we consider a more generic(fitting more physical properties) liquid crystals system (\ref{equa1})-(\ref{equa3}) and study its existence of weak solutions in both bounded domain and the whole space.   The main difficulties for solving the problem are dealing with pressure term and the nonlinear terms appeared in the stress tensors.  Like
Navier-Stokes equations, there are phenomenons of losing
compactness caused by the higher nonlinear terms of the equations.
Fortunately, we overcome those
difficulties by using much more technical methods such as
compensated compactness used in P. L. Lions' book \cite{lions2}, also to see E. Feireisl \cite{f1}, \cite{f2}.

The remaining part of this paper is organized as follows. Section \ref{2}, \ref{3}, \ref{4} are devoted to proof Theorem \ref{th1}. Using Theorem \ref{th1}, we prove Theorem \ref{th2} in Section \ref{5}.

\section{Approximate solutions}\label{2}

\setcounter{equation}{0}

In this section, similar to Eduard Feireisl did on Navier-Stokes equations(see \cite{f1}, \cite{f2}), we firstly construct an approximate problem:
\begin{align}
&\rho_t+div(\rho u)=\epsilon\Delta\rho,\label{approx1}\\
&(\rho u)_t+\mathrm{div}(\rho u\otimes u)+\nabla\left(p-\frac{1}{2}|\nabla d|^2-F(d)\right)+\nabla\cdot(\nabla d\odot\nabla d)\nonumber\\
&\qquad\qquad\qquad\quad\qquad+\nabla(\delta\rho^\beta)+\epsilon(\nabla u\cdot\nabla\rho)=div\widehat{\sigma},\label{approx2}\\
&\lambda_1d_t+\lambda_1u\cdot\nabla d-\lambda_1\Omega d+\Delta d-f(d)=0,\label{approx3}
\end{align}
complemented by the initial conditions:
\begin{eqnarray}
&&\rho(x,0)=\rho_{0}(x)\in C^2(\overline{D}),\ 0<\underline{\rho}\leq\rho_{0}\leq\overline{\rho}, \nabla \rho_0\cdot \vec{n}|_{\partial\Omega} =0,\label{approini1}\\
&&(\rho u)(x,0)=q_{0}(x),\ q_{0}\in C^2(\overline{D};\mathbb{R}^3),\label{approini2}\\
&&d(x,0)=d_{0}(x),\ d_{0}\in C^2(\overline{D};\mathbb{R}^3),\ |d_{0}(x)|=1\label{approini3}
\end{eqnarray}
and the boundary conditions:
\begin{eqnarray}
\nabla\rho\cdot \vec{n}|_{\partial D}=0,\label{approb1}\\
u|_{\partial D}=0,\label{approb2}\\
d|_{\partial D}=d_{0},\label{approb3}
\end{eqnarray}
where $\vec{n}$ is the out normal vector of $\partial D$.

For the Neumann problem  (\ref{approx1}), (\ref{approini1}) and (\ref{approb1}), by Lemma 2.2 of \cite{liu}, we have

\begin{lema}\label{prop1}
Let $D\subset \mathbb{R}^3$, be a bounded domain of $C^{2+\nu}$, $\nu >0$. Then there exists an operator $\rho=\mathcal{S}(u)$ solving the problem (\ref{approx1}), (\ref{approini1}) and (\ref{approb1}),
\begin{eqnarray}
\mathcal{S}:C([0,T];C^2(D))\mapsto C([0,T];C^{2+\nu}(D))\nonumber
\end{eqnarray}
 having the following properties:
\begin{itemize}
\item $\rho=\mathcal{S}(u)$ is the unique classical solution of (\ref{approx1}), (\ref{approini1}) and (\ref{approb1});
\item $\rho\geq\inf_{x\in\Omega}\rho(0,x)exp\left(-\int^t_0\|divu\|_\infty(s)ds\right),\ and\\
    \rho\leq\sup_{x\in\Omega}\rho(0,x)exp\left(\int^t_0\|divu\|_\infty(s)ds\right);$
\item $\sup_{t\in[0,T]}\|\rho\|_{W^{1,2}(D)}+\sqrt{\epsilon}\|\rho\|_{L^2(0,T;W^{2,2}(D))}\leq C(T,\rho_0,\|\nabla u\|_{L^\infty((0,T)\times D)});$
\item $\|\mathcal{S}(u^1)-\mathcal{S}(u^2)\|_{C([0,T];W^{1,2}(D))}\leq T^{\frac{1}{2}}C\|u^1-u^2\|_{L^\infty(0,T;W^{1,2}(D))}.$
\end{itemize}
\end{lema}

Next we consider the director movement equation (\ref{approx3}). Similarly as Lemma 2.3 in \cite{liu}, we have

\begin{lema}\label{prop2} There exists a mapping $R=R(u)$,
$$
R:C([0,T];(C^2(\overline{\Omega}))^3)\mapsto
C([0,T];C^\infty(\overline{\Omega}))\nonumber
$$
enjoying the following properties:

\begin{itemize}
\item  $d=R(u)$ is the smooth solution of the problem (\ref{approx3}), (\ref{approini3}) and (\ref{approb3}):
\item $\|R(u)\|_{L^\infty(0,T;H^1(D))}+\|R(u)\|_{L^2(0,T;H^2(D))}
\leq C(T,\ d_{0},\ u);$
\item $\|R(u^1)-R(u^2)\|_{L^\infty([0,T];H^{1}(D))}\leq T^{\frac{1}{2}}C\|u^1-u^2\|_{L^\infty(0,T;H^{1}(D))}.$
\end{itemize}
\end{lema}

\subsection{The Faedo-Galerkin approximations}

In the subsection we solve the equation (\ref{approx2}) with initial condition (\ref{approini2}) by Faedo-Galerkin approximate.

Let $\{\phi_i\}^\infty_{i=1}$ be the orthogonal basis of $H^1_0(D)$, which satisfies:
$$
-\Delta \phi_i=a_i\phi_i, \quad \text{in}\quad\Omega, \quad \phi_i|_{\partial D}=0.
$$
Here $a_i,\ (i=1,2\cdots)$ is the eigenvalue of the operate $-\Delta$, and $0<a_1\leq a_2\leq\cdots a_n\leq\cdots$, $a_n\rightarrow\infty,\ n\rightarrow\infty$. We consider the finite dimensional space
\begin{eqnarray}
X_n=span\{\phi_i\}^n_{i=1},\quad\quad n=1,2,3\cdots\nonumber
\end{eqnarray}
And $X_n$ is Hilbert space equipped with norm given by scalar product of $L^2$.
We shall look for the approximate solution $u=u_n\in C([0,T];X_n)$,
satisfying the following integral equation
\begin{align}
\int_{D}\rho u\cdot\eta -\int_{D} q_{0,\delta}\cdot\eta &=\int^t_0\!\!\!\int_{D}\left\{\!\!\left(p+\delta\rho^\beta-\frac{1}{2}\left(|\nabla d|^2+F(d)\right)\!\!\right)div\eta-\epsilon(\nabla u\nabla\rho)\eta\!\right\} \nonumber\\
&+\int^t_0\int_{D}\left(\rho u\otimes u+\nabla d\otimes\nabla d-\widehat{\sigma}\right):\nabla \eta ,\label{appe5}
\end{align}
for any $\eta\in X_n,\ t\in [0,T]$.
Define a map $M[\rho]$.
\begin{eqnarray}
M[\rho]:X_n\mapsto X^*_n,\quad \left<M[\rho]u,w\right>\equiv\int_{D}\rho u\cdot wdx,\quad u,w\in X_n,\nonumber
\end{eqnarray}
where $X^*_n$ is the dual space of $X_n$. Since $\rho$ has a positive lower bound, then $M[\rho]$ is invertible, and satisfies
\begin{eqnarray}
\|M^{-1}[\rho]\|_{L(X^*_n;X_n)}\leq\frac{1}{\inf_{D}\rho},\nonumber
\end{eqnarray}
\begin{eqnarray}
M^{-1}[\rho^1]-M^{-1}[\rho^2]=M^{-1}[\rho^2]\left(M[\rho^1]-M[\rho^2]\right)M^{-1}[\rho^1].\nonumber
\end{eqnarray}
Obviously from above{\color{blue},} we have
\begin{eqnarray}
\|M^{-1}[\rho^1]-M^{-1}[\rho^2]\|_{L(X^*_n;X_n)}\leq C(n,\rho^1,\rho^2)\|\rho^1-\rho^2\|_{L^\infty(D)}.\nonumber
\end{eqnarray}
%Here $\eta=\inf\{\rho^1,\rho^2\}$.
The equation (\ref{appe5}) can be rewritten as
\begin{eqnarray}
u(t)=M^{-1}[\rho(t)]\left(q^*_0+\int^t_0\mathcal{N}(\mathcal{S}(u)(s),u(s),R(u)(s))ds\right).\label{gl2}
\end{eqnarray}
Here $q^*_0\in X^*_n,\ \mathcal{N}\in X^*_n$ satisfy
\begin{align}
\left<q^*_0,\psi\right>=\int_{D} p_{0,\delta}\cdot\psi,\nonumber
\end{align}
\begin{align}
&\left<\mathcal{N}\left(\rho(s),u(s),d(s)\right),\psi\right>\nonumber\\
&=\int^t_0\int_{D}\left(p+\delta\rho^\beta-\frac{1}{2}\left(|\nabla d|^2+F(d)\right)\right)div\psi \nonumber\\
&+\int^t_0\int_{D}\left[-\epsilon(\nabla u\nabla\rho)\psi+(\rho u\otimes u+\nabla d\otimes\nabla d-\widehat{\sigma}):\nabla\psi\right].\nonumber
\end{align}
Next we deduce the energy estimates of (\ref{approx1})-(\ref{approx3}) in finite dimensional space $X_n$. Let $(\mathcal{S}(u),u,R(u))$
be a solution of (\ref{approx1})-(\ref{approx3}) for  $u\in X_n$. We rewrite (\ref{appe5}) as follows:
\begin{align}
\int_{D}(\rho u)_t\cdot\eta =&\int^t_0\int_{D}\left\{\left(p+\delta\rho^\beta-\frac{1}{2}\left(|\nabla d|^2+F(d)\right)\right)div\eta-\epsilon(\nabla u\nabla\rho)\eta\right\} \nonumber\\
+&\int^t_0\int_{D}(\rho u\otimes u+\nabla d\otimes\nabla d-\widehat{\sigma}):\nabla \eta .\label{appforget}
\end{align}
Let $\eta=u$ at time $t$. We immediately get
\begin{eqnarray}
&&\frac{d}{dt}\int_{D}\frac{1}{2}\rho|u|^2\nonumber\\&=&\int_{D}\left[\left(p+\delta\rho^\beta\right)divu+f(d)((u\cdot\nabla)d)-\Delta d((u\cdot\nabla)d)-\widehat{\sigma}:\nabla u\right].\label{appe6}
\end{eqnarray}
Equation (\ref{approx1}) together with (\ref{appe6}) yields
\begin{eqnarray}
&&\frac{d}{dt}\int_{D}\left[\frac{1}{2}\rho|u|^2+\frac{1}{\gamma-1}p+\frac{\delta}{\beta-1}\rho^\beta\right]\nonumber\\
&&=\int_{D} \left[f(d)((u\cdot\nabla)d)-\Delta d((u\cdot\nabla)d)-\widehat{\sigma}:\nabla u\right]\nonumber\\
&&+\int_{D} \left(a\gamma\epsilon\Delta\rho\rho^{\gamma-1}+\delta\beta\epsilon\Delta\rho\rho^{\beta-1}\right).\nonumber
\end{eqnarray}
%The relation (\ref{ap3}) also yields,
%\begin{eqnarray}
%\frac{d}{dt}\int_{D}\frac{1}{2}|\nabla d|^2=\int_{D}\frac{1}{\lambda_1}(\Delta d-f(d))\Delta d+(u\cdot\nabla)\Delta d-(\Omega d)\Delta d.\nonumber
%\end{eqnarray}
%Also we can get
%\begin{eqnarray}
%\frac{d}{dt}\int_{D} F(d)=\int_{D} f(d)\cdot(-(u\cdot\nabla)d+\Omega d-\frac{1}{\lambda_1}\Delta d+\frac{1}{\lambda_1}f(d)).\nonumber
%\end{eqnarray}
Integrating above equation in time to obtain
\begin{eqnarray}
&&\int_{D}\left[\frac{1}{2}\rho|u|^2+\frac{1}{2}|\nabla d|^2+\frac{1}{\gamma-1}p+\frac{\delta}{\beta-1}\rho^\beta+F(d))\right](\tau)\nonumber\\
&&+\int^\tau_0\int_{D}\left[\mu_4|A|^2-\lambda_1|N|^2+\mu_7(divu)^2\right]\nonumber\\
&&+\int^\tau_0\int_{D}\epsilon\left(a\gamma|\nabla\rho^\frac{\gamma}{2}|^2+\delta\beta|\nabla\rho^\frac{\beta}{2}|^2\right)\nonumber\\
&&\leq E_{0,\delta},\label{g14}
\end{eqnarray}
where $E_{0,\delta}$ is
\begin{eqnarray}
E_{0,\delta}=\int_{D}\left[\frac{1}{2}\frac{q^2_{0,\delta}}{\rho_{0,\delta}}+\frac{1}{2}|\nabla d_{0,\delta}|^2+\frac{1}{\gamma-1}p_{0,\delta}+\frac{\delta}{\beta-1}\rho^\beta_{0,\delta}+F(d_{0,\delta})\right].\nonumber
\end{eqnarray}

From the energy law and (\ref{g14}), using the standard method in \cite{lz},
the integral equation (\ref{appe5}) can be solve in any interval $[0,T]$.

By energy law again, we have the following lemma describing  the
information of the approximate solution $(\rho_n, u_n, d_n)$.
\begin{prop}\label{prop3}
For any fixed $n$, and $T<\infty$, there exists $(\rho,u,d)$ to solve problem (\ref{approx1})-(\ref{approb3}). And we have
\begin{align}
&\|\sqrt{\rho_n}u_n\|_{L^\infty(0,T;L^2(D))}\leq E_{0,\delta},\qquad\|\rho_n\|_{L^\infty(0,T;L^\gamma(D))}\leq E_{0,\delta},\label{appe7}\\
&\delta\|\rho_n\|^\beta_{L^\infty(0,T;L^\beta(D))}\leq E_{0,\delta},\qquad\|u_n\|_{L^2(0,T;W^{1,2}(D))}\leq E_{0,\delta},\\
&\|N_n\|_{L^2((0,T)\times D)}\leq E_{0,\delta},\qquad\|d_n\|_{L^\infty(0,T;W^{1,2}(D))}\leq E_{0,\delta},\label{appe8}\\
&\|d_{nt}\|_{L^2(0,T;L^{\frac{3}{2}}(D))}\leq C(E_{0,\delta}),\qquad\|d_n\|_{L^2(0,T;W^{2,2}(D))}\leq E_{0,\delta},\label{appe9}\\
&\epsilon\int^T_0\int_{D}|\nabla \rho_n|^2\leq C(T,E_{0,\delta}),\qquad\|\rho_n\|_{L^{\frac{4\beta}{3}}((0,T)\times D)}\leq C(T,\epsilon,E_{0,\delta}).\label{appe10}
\end{align}
\end{prop}
\begin{proof}
(\ref{appe7})-(\ref{appe8}) can be directly obtained from (\ref{g14}). So we only need to consider (\ref{appe9}), (\ref{appe10}). The first term of (\ref{appe9}) is due to H$\ddot{o}$ld inequality. By elliptic estimates, we have
\begin{align}
\|\nabla^2 d_n\|_{L^2((0,T)\times D)}&\leq\!C(\|\Delta d_n\|_{L^2((0,T)\times D)}\!+\!\| d_n\|_{L^2((0,T)\times D)}\!+\!\|\nabla^2 d_{0,\delta}\|_{L^2((0,T)\times D)})\nonumber\\
&\leq\!\lambda_1\|N_n\|_{L^2((0,T)\times D)}+\|f(d_n)\|_{L^2((0,T)\times D)}+C\nonumber\\
&\leq\!C(E_{0,\delta},|D|).
\end{align}
 Using equation (\ref{approx1}), we get
\begin{eqnarray}
\frac{d}{dt}\int_{D}\rho_n^2+2\epsilon\int_{D}|\nabla\rho_n|^2=\int_{D}|\rho_n|^2divu_n.\nonumber
\end{eqnarray}
and then integrate to obtain
\begin{align}
\int_{D}|\rho_n|^2+\epsilon\int^T_0\int_{D}|\nabla\rho_n|^2&\leq\int_{D}|\rho_{0,\delta}|^2+\int^T_0\int_{D}|\rho_n|^2|\nabla u_n|\nonumber\\
&\leq\int_{D}|\rho_{0,\delta}|^2+\|\rho_n\|^2_{L^\infty(0,T;L^4(D))}\int^T_0\int_{D}|\nabla u_n|^2\nonumber\\
&\leq C(T,E_{0,\delta}),\qquad if\ \beta\geq4.
\end{align}
By inequality (\ref{g14}) and $\epsilon\delta\beta\rho_n^{\beta-2}|\nabla\rho|^2=\epsilon\delta|\nabla(\rho_n^{\frac{\beta}{2}})|^2$, we get
\begin{eqnarray}
\|\sqrt{\epsilon\delta}\rho^{\frac{\beta}{2}}_n\|_{L^2(0,T;W^{1,2}(D))}\leq E_{0,\delta}.\nonumber
\end{eqnarray}
By Sobolev's inequality, we have
\begin{eqnarray}
&&\rho^\beta_n\in L^1(0,T;L^3(D)), \quad N=3,\label{t1}\\
&&\rho^\beta_n\in L^\infty(0,T;L^1(D)). \label{t2}
\end{eqnarray}
(\ref{t1}) is equivalent to
\begin{eqnarray}
\rho_n\in L^\beta(0,T;L^{3\beta}(D)).\label{t4}
\end{eqnarray}
Then
\begin{align}
\|\rho_n\|^{\frac{4\beta}{3}}_{L^{\frac{4\beta}{3}((0,T)\times D)}}&\leq\int^T_0\|\rho_n\|^{\frac{\beta}{2}}_{L^{3\beta}(D)}\|\rho_n\|^{\frac{5\beta}{6}}_{L^\beta(D)}\nonumber\\
&\leq T^{\frac{1}{2}}\|\rho_n\|^{\frac{5\beta}{6}}_{L^\infty(0,T;L^\beta(D))}\|\rho_n\|^{\frac{\beta}{2}}_{L^\beta(0,T;L^{3\beta}(D))}.\label{t5}
\end{align}
Here we use the interpolation of (\ref{t2}) and (\ref{t4}).
\end{proof}

\subsection{The process of $n\rightarrow\infty$}

In this section, we let $n\rightarrow\infty$ in the sequence $\{(\rho_n,u_n,d_n)\}$. The following compactness theorem due to J. Lions(see \cite{simon} and \cite{temam}).
\begin{lema}\label{lemma4}
Let $X_0,\ X,\ X_1$ be three Banach spaces such that $X_0\hookrightarrow X\hookrightarrow X_1$, $X_0\hookrightarrow X$ is compact, and $X_0,X_1$ are reflexive. And define
$Y=\{v\in L^{\alpha_0}(0,T;X_0);\\ \frac{dv}{dt}\in L^{\alpha_1}(0,T;X_1)\}$ with norm
\begin{eqnarray}
\|v\|_Y=\|v\|_{L^{\alpha_0}(0,T;X_0)}+\|v_t\|_{L^{\alpha_1}(0,T;X_1)},\nonumber
\end{eqnarray}
$1<\alpha_i<\infty,i=1,2$. Then $Y\hookrightarrow L^{\alpha_0}(0,T;X)$ is compact.\\
Moreover, if $\alpha_0=\infty$,
\begin{eqnarray}
Y\hookrightarrow C([0,T];X)\ is\ compact.\nonumber
\end{eqnarray}
\end{lema}
%\begin{remark}
%For the sake of simplicity, we denote as $Y$ as $L^{\alpha_0}(0,T;{X_0}_{weak})$.
%\end{remark}
Using equation (\ref{approx1}), we have
\begin{eqnarray}
\int_{D}\rho_{nt}\varphi =\int_{D}\left(\epsilon\Delta\rho_n\varphi - div(\rho_nu_n)\varphi \right)=\int_{D}\left(-\epsilon\nabla\rho_n\cdot\nabla\varphi +\rho_nu_n\cdot\nabla\varphi\right),\nonumber
\end{eqnarray}
where $\varphi\in W^{1,p'}_0(D)$, $2<p'<\infty$. So we obtain
\begin{align}
\|\frac{\partial}{\partial t}\rho_n\|_{W^{-1,p}(D)}&\leq\|\rho_nu_n\|_{L^p(D)}+\epsilon\|\nabla\rho_n\|_{L^p(D)}\nonumber\\
&\leq\|\sqrt{\rho_n}u_n\|_{L^2(D)}\|\rho_n\|^{\frac{1}{2}}_{L^{\frac{p}{2-p}}(D)}+\epsilon\|\nabla\rho_n\|_{L^p(D)},\quad p<2,\nonumber
\end{align}
and then
\begin{eqnarray}
\int^T_0\|\frac{\partial}{\partial t}\rho_n\|^p_{W^{-1,p}(D)}dt\leq C(T,\rho_{0,\delta},p,d_0).\nonumber
\end{eqnarray}
By Sobolev embedding theory, we have $H^1(D)\hookrightarrow L^k(D)\hookrightarrow W^{-1,p}(D)$ for $\frac{6}{5}\leq k<6$, and $H^1(D)\hookrightarrow L^k(D)$ compactly.
Using {Lemma \ref{lemma4}}, we obtain
\begin{eqnarray}
\rho_n\rightarrow\rho\ in\ L^{\gamma}((0,T)\times D),\ \ \rho_n\rightarrow\rho\ in\ L^{\beta}((0,T)\times D),\quad \frac{3}{2}<\gamma,\beta\leq6.\label{2.2}
\end{eqnarray}
Moreover, we have
\begin{eqnarray}
\rho^\beta_n\rightarrow\rho^\beta\quad in\ L^1((0,T)\times D),\quad \rho^\gamma_n\rightarrow\rho^\gamma\quad in\ L^1((0,T)\times D).\nonumber
\end{eqnarray}
Similarly to $\rho$, we have
\begin{eqnarray}
d_n\rightarrow d\quad in\ L^2(0,T;W^{1,p}(D)),\quad d_n\rightarrow d\quad in\  C([0,T];L^p(D)),\ 1\leq p<6.\label{max1}
\end{eqnarray}
Then using $|d|^2\leq1$, we obtain
\begin{eqnarray}
d_n\rightarrow d\quad in\ L^{k'}((0,T)\times D)\quad 1<k'<\infty.\nonumber
\end{eqnarray}
It is easy to get
\begin{eqnarray}
\rho_nu_n\rightharpoonup^*\rho u\quad in\ L^\infty(0,T;L^{\frac{2\gamma}{\gamma+1}}(D)).\label{appe11}
\end{eqnarray}
As $d_n$ has good regularities, it is easy to deduce $\widehat{\sigma}_n\rightharpoonup\widehat{\sigma}$ in $D'((0,T)\times D)$ as well as the weak convergence of other terms except for the following two terms
\begin{eqnarray}
\rho_nu_n\otimes u_n\rightharpoonup\rho u\otimes u,\quad \nabla u_n\nabla\rho_n\rightharpoonup\nabla u\nabla\rho \quad in\ D'((0,T)\times D).\nonumber
\end{eqnarray}

In order to continue, we need the following lemma (see \cite{f2} Lemma 7.7.5).
\begin{lema}\label{lemma5}
Suppose $\rho_n$ is a solution of (\ref{approx1}) supplement with the boundary conditions (\ref{approb1}) corresponding to $u_n$. Then there exist $r>1,q>2$, such that $\partial_t\rho_n,\ \Delta\rho_n$ are bounded in $L^r((0,T)\times D)$, $\nabla\rho_n$ is bounded in $L^q(0,T;L^2(D))$ independently of $n$. Accordingly, the limit function $\rho$ belongs to the same class and satisfies equation (\ref{approx1}) $a.e.\ on\ (0,T)\times D$ together with the boundary conditions (\ref{approb1}) in the sense of trace.
\end{lema}
Since $\rho_nu_n$ satisfies (\ref{appforget}), the relation (\ref{appe11}) can be strengthened as the following term
\begin{eqnarray}
\rho_nu_n\rightharpoonup\rho u\ in\ L^\infty(0,T;L^{\frac{2\gamma}{\gamma+1}}(D)).\label{n2}
\end{eqnarray}
Observing equation (\ref{approx2}), we have
\begin{eqnarray}
\partial_t(\rho_nu_n)\in L^2(0,T;H^{-s}(D)),\ s\geq \frac{5}{2}.
\end{eqnarray}
Using Lemma \ref{lemma4}, the above estimates are enough to show
\begin{eqnarray}
\rho_nu_n\rightarrow\rho u\quad in\ C([0,T];W^{-1,2}(D)),\nonumber
\end{eqnarray}
and then
\begin{eqnarray}
\rho_nu_n\otimes u_n\rightharpoonup\rho u\otimes u\quad in\ D'((0,T)\times D).\nonumber
\end{eqnarray}
Next we consider $\nabla u\nabla\rho$. By virtue of {Lemma \ref{lemma5}}, we have
\begin{eqnarray}
(\rho-\rho_n)_t-\epsilon\Delta(\rho-\rho_n)=div((\rho-\rho_n)u+\rho_n(u_n-u)),
\end{eqnarray}
Multiplying above equation with $\rho-\rho_n$, we obtain
\begin{align}
&\int_{D}|\rho-\rho_n|^2+2\epsilon\int^T_0\int_{D}|\nabla(\rho-\rho_n)|^2\nonumber\\
&=\int^T_0\int_{D}\left[|\rho-\rho_n|^2divu+(\rho-\rho_n)\nabla(\rho-\rho_n)\cdot u\right]\nonumber\\
&+\int^T_0\int_{D}\left[\rho_n(\rho-\rho_n)div(u_n-u)+(\rho-\rho_n)\nabla\rho_n\cdot(u_n-u)\right].\nonumber
\end{align}
By choosing a suitable $\beta$ in equation (\ref{2.2}), the above equation yields
\begin{eqnarray}
\nabla\rho_n\rightarrow\nabla\rho\quad in\ L^2((0,T)\times D),\nonumber\\
\rho_n\rightarrow\rho\quad in\ L^\infty(0,T;L^2(D)).\nonumber
\end{eqnarray}
Therefore we obtain
\begin{eqnarray}
\nabla u_n\nabla\rho_n\rightharpoonup\nabla u\nabla\rho \quad in\ D'((0,T)\times D).\nonumber
\end{eqnarray}
Summing up the work of this section, we have the following result.
\begin{prop}\label{prop4}
Suppose $\beta>max\{4,\gamma\}$, $\Omega\subset \mathbb{R}^3$ is a bounded domain in $C^{2+\alpha}$. Let $\rho_{0,\delta},\ q_{0,\delta},\ d_{0,\delta}$ satisfy (\ref{approini1})-(\ref{approb3}). Then there exists a weak solution $(\rho,u,d)$ to the problem (\ref{approx1})-(\ref{approx3}), such that
\begin{align}
&\|\sqrt{\rho}u\|_{L^\infty(0,T;L^2(D))}\leq C(E_{0,\delta}),\qquad\|\rho\|_{L^\infty(0,T;L^\gamma(D))}\leq C(E_{0,\delta}),\nonumber\\
&\delta\|\rho\|^\beta_{L^\infty(0,T;L^\beta(D))}\leq C(E_{0,\delta}),\qquad\|u\|_{L^2(0,T;H^1_0(D))}\leq C(E_{0,\delta}),\nonumber\\
&\|d\|_{L^\infty(0,T;H^2(D))}\leq C(E_{0,\delta}),\qquad\|u\|_{L^2(0,T;H^2(D))}\leq C(E_{0,\delta}),\nonumber\nonumber\\
&\|N\|_{L^2((0,T)\times D)}\leq C(E_{0,\delta}),\qquad\|d_t\|_{L^2(0,T;L^{\frac{3}{2}}(D))}\leq C(E_{0,\delta}),\nonumber\\
&\epsilon\|\nabla\rho\|^2_{L^2((0,T)\times D)}\leq C(E_{0,\delta}),\qquad\|\rho\|_{L^{\frac{4\beta}{3}}((0,T)\times D)}\leq C(E_{0,\delta}).
\nonumber
\end{align}
and {Lemma \ref{lemma5}} holds. Moreover, we have the following energy law:
\begin{align}
&\frac{d}{dt}\int_{D}\left[\frac{1}{2}\rho|u|^2+\frac{1}{2}|\nabla d|^2+\frac{1}{\gamma-1}p+\frac{\delta}{\beta-1}\rho^\beta+F(d)\right]\nonumber\\
&=-\int_{D}\left[\mu_4|A|^2-\lambda_1|N|^2+\mu_7(divu)^2+\epsilon(a\gamma\rho^{\gamma-2}+\delta\beta\rho^{\beta-2})|\nabla\rho|^2\right].\label{l3}
\end{align}
\end{prop}

\section{Taking limit  $\epsilon\rightarrow0$}\label{3}

\setcounter{equation}{0}

We introduce an operator $B=[B_1,B_2,B_3]$ corresponding in a certain sense to the inverse of $div_x$.
\begin{eqnarray}
div_xv=g-\frac{1}{|D|}\int_{D} gdx\quad on\ D,\quad v|_{\partial D}=0.\label{fdiv}
\end{eqnarray}
It can be shown (see \cite{bog} or \cite{galdi} Theorem 10.3.3) that (\ref{fdiv}) admits an operator $B:g\mapsto v$ enjoying the following properties:
\begin{itemize}
\item $B$ is a bounded linear operator from $L^p(D)$ into $W^{1,p}_0(D)$ for any $1<p<\infty$,
\item the function $v=B[g]$ solves the problem (\ref{fdiv}),
\item if the function $g\in L^p(D)$ can be written in the form $g=div_xh$ where $h\in L^r(D)$,\\$h\cdot \vec{n}=0$ on $\partial D$, then
    $$\|B[g]\|_{L^r(D)}\leq c(p,r)\|h\|_{L^r(D)}.$$
\end{itemize}
 Let $\psi(t)\in C_0^{\infty}(0,T)$, $0\leq\psi\leq1,\ m_0=\frac{1}{|\Omega|}\int_{D} \rho_\epsilon dx$. Taking $\varphi=\psi(t)B[\rho_\epsilon-m_0]$ as a test function of (\ref{approx2}), a direct calculation yields
\begin{align}
&\int^T_0\int_{D}\left\{-\rho_\epsilon u_\epsilon\psi_tB[\rho_\epsilon-m_0]-\rho_\epsilon u_\epsilon\psi B[\epsilon\Delta\rho_\epsilon]+\rho_\epsilon u_\epsilon\psi B_i[div(\rho_\epsilon u_\epsilon)]\right\}\nonumber\\
&+\int^T_0\int_{D}\left\{-(p_\epsilon+\delta\rho_\epsilon^\beta) divB[\rho_\epsilon-m_0]-\psi\rho_\epsilon u_\epsilon\otimes u_\epsilon:\nabla\right\} B[\rho_\epsilon-m_0]
\nonumber\\
&+\int^T_0\int_{D}\left\{\frac{1}{2}(|\nabla d_\epsilon|^2+F(d_\epsilon))divB[\rho_\epsilon-m_0]-(\nabla d_\epsilon\otimes\nabla d_\epsilon):\nabla B[\rho_\epsilon-m_0]\right\}\nonumber\\
&+\int^T_0\int_{D}\left\{\widehat{\sigma}_\epsilon:\nabla B[\rho_\epsilon-m_0]+\epsilon\psi(\nabla u_\epsilon\nabla\rho_\epsilon)B[\rho_\epsilon-m_0]\right\}=0.\nonumber
\end{align}
Then we have
\begin{align}
&\int^T_0\int_{D}\psi\left(a\rho^{\gamma+1}_\epsilon+\delta\rho^{\beta+1}_\epsilon\right)\nonumber\\
&=\int^T_0\int_{D}\left\{\psi m_0(a\rho^\gamma_\epsilon+\delta\rho^\beta_\epsilon)-\psi_t\rho_\epsilon u_\epsilon B[\rho_\epsilon-m_0]-\psi\rho_\epsilon u_\epsilon B[\epsilon\Delta\rho_\epsilon]\right\}\nonumber\\
&+\int^T_0\int_{D}\left\{\psi\rho_\epsilon u_\epsilon B[div(\rho_\epsilon u_\epsilon)]-\psi\rho_\epsilon u_\epsilon\otimes u_\epsilon:\nabla B[\rho_\epsilon-m_0]\right\}\nonumber\\
&+\int^T_0\int_{D}\left\{\frac{1}{2}(|\nabla d_\epsilon|^2+F(d_\epsilon))(\rho_\epsilon-m_0)-\psi(\nabla d_\epsilon\otimes\nabla d_\epsilon):\nabla B[\rho_\epsilon-m_0]\right\}\nonumber\\
&+\int^T_0\int_{D}\left\{\psi\widehat{\sigma}_\epsilon:\nabla B[\rho_\epsilon-m_0]+\epsilon\psi(\nabla u_\epsilon\nabla\rho_\epsilon)B[\rho_\epsilon-m_0]\right\}\nonumber\\
&=\sum^9_{i=1}I_i.\nonumber
\end{align}
We estimate each $I_i$ as follows:
\begin{eqnarray}
I_1=\int^T_0\int_{D}\psi m_0\left(a\rho^\gamma_\epsilon+\delta\rho^\beta_\epsilon\right)\leq C(E_{0,\delta}),\nonumber
\end{eqnarray}
\begin{align}
I_2&=\int^T_0\int_{D}-\psi_t\rho_\epsilon u_\epsilon B[\rho_\epsilon-m_0]\nonumber\\
&\leq \int^T_0|\psi_t|\|\sqrt{\rho_\epsilon}u_\epsilon\|_{L^2(D)}\|\sqrt{\rho_\epsilon}\|_{L^2(D)}\|B\|_{L^\infty(D)}dt \nonumber\\
&\leq C(E_{0,\delta})\int^T_0|\psi_t|dt,\nonumber
\end{align}
\begin{align}
I_3\!=\!\int^T_0\!\!\int_{D}\psi\rho_\epsilon u_\epsilon B[\epsilon\Delta\rho_\epsilon]&\!\leq\!\sqrt{\epsilon}\int^T_0|\psi|\sqrt{\epsilon}
\|\nabla\rho_\epsilon\|_{L^2(D)}\|\rho_\epsilon\|_{L^3(D)}\|u_\epsilon\|_{L^6(D)}dt\nonumber\\
&\leq C(E_{0,\delta}),\nonumber
\end{align}
\begin{align}
I_4=\int^T_0\int_{D}\psi\rho_\epsilon u_\epsilon B[div(\rho_\epsilon u_\epsilon)]&\leq
\int^T_0\psi\|\rho_\epsilon u_\epsilon\|_{L^2(D)}\|\rho_\epsilon u_\epsilon\|_{L^2(D)}dt\nonumber\\
&\leq C(E_{0,\delta}),\nonumber
\end{align}
\begin{align}
I_5&=\int^T_0\int_{D}-\psi\rho_\epsilon u_\epsilon\otimes u_\epsilon:\nabla B[\rho_\epsilon-m_0]\nonumber\\
&\leq\int^T_0\psi\|\rho_\epsilon\|_{L^3(D)}
\|u_\epsilon\|^2_{L^6(D)}\left(\|\rho_\epsilon\|_{L^3(D)}+C\right)dt\nonumber\\
&\leq C(E_{0,\delta}),\nonumber
\end{align}
\begin{align}
I_6=\int^T_0\int_{D}\frac{1}{2}\psi\left(|\nabla d_\epsilon|^2+F(d_\epsilon)\right)(\rho_\epsilon-m_0)&\leq
C(E_{0,\delta}),\nonumber\\
I_7=\int^T_0\int_{D}-\psi(\nabla d_\epsilon\otimes\nabla d_\epsilon):\nabla B[\rho_\epsilon-m_0]&\leq C\int^T_0\int_{D}|\rho_\epsilon-m_0||\nabla d_\epsilon|^2\nonumber\\
&\leq C(E_{0,\delta}),\nonumber
\end{align}
\begin{align}
I_8&=\int^T_0\int_{D}\psi\widehat{\sigma}_\epsilon:\nabla B[\rho_\epsilon-m_0]\nonumber\\
&\leq C\int^T_0\int_{D}\left(|d_\epsilon||N_\epsilon||\rho_\epsilon-m_0|+|\nabla u_\epsilon||\rho_\epsilon-m_0|\right)\nonumber\\
&\leq C\int^T_0\!\!\left(\|d_\epsilon\|_{L^3(D)}\|N_\epsilon\|_{L^2(D)}\|\rho_\epsilon-m_0\|_{L^6(D)}\!+\!\|\nabla u_\epsilon\|_{L^2(D)}\|\rho_\epsilon-m_0\|_{L^2(D)}\right)dt\nonumber\\
&\leq C(E_{0,\delta})\nonumber
\end{align}
and
\begin{align}
I_9&=\epsilon\int^T_0\int_{D}\psi(\nabla u_\epsilon\nabla\rho_\epsilon)B[\rho_\epsilon-m_0]\nonumber\\
&\leq\sqrt{\epsilon}\int^T_0\sqrt{\epsilon}\|\nabla \rho_\epsilon\|_{L^2(D)}\|\nabla u_\epsilon\|_{L^2(D)}\|B[\rho_\epsilon-m_0]\|_{L^\infty(D)}dt\nonumber\\
&\leq C(E_{0,\delta}).\nonumber
\end{align}
Summing up above estimates, we have the following lemma:
\begin{lema}\label{lemma6}
Let $\rho_\epsilon,u_\epsilon,d_\epsilon$ be the solution to problem (\ref{approx1})-(\ref{approb3}). Then there exists a constant $C=C(E_{0,\delta})$ which is independent of $\epsilon$, such that
\begin{eqnarray}
\|\rho_\epsilon\|_{L^{\gamma+1}((0,T)\times D)}+\|\rho_\epsilon\|_{L^{\beta+1}((0,T)\times D)}\leq C(E_{0,\delta}).\label{t3}
\end{eqnarray}
\end{lema}
%Next we will let $\epsilon\rightarrow0$.
Due to the {Proposition \ref{prop4}}, we have
\begin{align}
&\int_0^T\int_D\epsilon\varphi\Delta\rho_\epsilon=\int_0^T\int_D-\epsilon\nabla \varphi\nabla\rho_\epsilon\nonumber\\
&\leq\sqrt{\epsilon}(\sqrt{\epsilon}\|\nabla\rho_\epsilon\|_{L^2((0,T)\times D)})\|\nabla \varphi\|_{L^2((0,T)\times D)}\rightarrow0,\quad as\ \epsilon\rightarrow0.\label{fe1}
\end{align}
The relation (\ref{fe1}) yields
\begin{eqnarray}
\epsilon\Delta\rho_\epsilon\rightarrow0\quad in\ L^2(0,T;H^{-1}(D)).\nonumber
\end{eqnarray}
By virtue of {Proposition \ref{prop4}} and (\ref{t3}), we get
\begin{eqnarray}
&&\rho_\epsilon\rightarrow\rho\quad in\ C([0,T];L^\gamma_{weak}(D)),\nonumber\\
&&\rho_\epsilon\rightharpoonup\rho\quad in\ L^{\gamma+1}((0,T)\times D).\nonumber
\end{eqnarray}
As we have $u_\epsilon\rightharpoonup u\ in\ L^2(0,T;H^{1}_0(D))$,
the following term holds
\begin{eqnarray}
\rho_\epsilon u_\epsilon\rightarrow\rho u\quad in\ C([0,T];L^{\frac{2\gamma}{\gamma+1}}_{weak}(D)).\nonumber
\end{eqnarray}
Then it is natural to obtain
\begin{eqnarray}
\rho_\epsilon u_\epsilon\otimes u_\epsilon\rightharpoonup\rho u\otimes u\quad in\ D'((0,T)\times D).\nonumber
\end{eqnarray}
Using {Lemma \ref{lemma4}} and
\begin{eqnarray}
&&d_\epsilon\rightharpoonup d\ in\ L^2(0,T;H^2(D)),\nonumber\\
&&d_\epsilon\rightharpoonup d\ in\ L^\infty(0,T;H^1(D)),\nonumber
\end{eqnarray}
we get
\begin{align}
&d_\epsilon\rightarrow d\ in\ L^2(0,T;W^{1,p}(D)),\nonumber\\
&d_\epsilon\rightarrow d\ in\ C([0,T];L^p(D)),\qquad 1\leq p<6.\nonumber
\end{align}

Thus one can prove
\begin{align}
|\nabla d_\epsilon|^2+F(d_\epsilon)\rightarrow|\nabla d|^2+F(d)\quad in\ L^1((0,T)\times D), \nonumber\\
\nabla d_\epsilon\otimes\nabla d_\epsilon\rightarrow\nabla d\otimes\nabla d\quad in\ L^1((0,T)\times D),\nonumber\\
d_\epsilon N_\epsilon\rightharpoonup dN\quad in\ L^1((0,T)\times D).\nonumber
\end{align}
Denoting $\overline{p}:\ a\rho^\gamma_\epsilon+\delta\rho^\beta_\epsilon\rightharpoonup \overline{p}\ in\ L^{\frac{\beta+1}{\beta}}((0,T)\times D)$,
the limit equations read
\begin{align}
&\rho_t+\mathrm{div}(\rho u)=0\label{fequ1}\\
&(\rho u)_t+\mathrm{div}(\rho u\otimes u)+\nabla(\overline{p}-\frac{1}{2}|\nabla d|^2-F(d))+\nabla\cdot(\nabla d\odot\nabla d)=\mathrm{div}\widehat{\sigma},\label{fequ2}\\
&\lambda_1d_t+\lambda_1u\cdot\nabla d-\lambda_1\Omega d+\Delta d-f(d)=0,\label{fequ3}
\end{align}
where $\widehat{\sigma}=\mu_3(N\otimes d-d\otimes N)+\mu_4A+\mu_7tr(A)I$.
The next part will contribute to prove $\overline{p}=a\rho^\gamma+\delta\rho^\beta$.\\

Let
$0\leq\phi_m\leq1,\phi_m=1\ in\ \{x|dist[x,\partial D]\geq\frac{1}{m}\}$, $\varphi\in C_0^\infty((0,T)\times \mathbb{R}^3)$.
Then
\begin{eqnarray}
\int^T_0\int_{\mathbb{R}^3}\rho\varphi_t=\int^T_0\int_{\mathbb{R}^3}\left[\rho(\phi_m\varphi)_t+\rho(1-\phi_m)\varphi_t\right],\nonumber
\end{eqnarray}
\begin{eqnarray}
\int^T_0\int_{\mathbb{R}^3}\rho u\nabla\varphi =\int^T_0\int_{\mathbb{R}^3}\left[\rho u\nabla(\phi_m\varphi)+\rho u(1-\phi_m)\nabla\varphi-\rho u\varphi\nabla\phi_m\right].\nonumber
\end{eqnarray}
Passing to the limit for $m\rightarrow\infty$, we have
\begin{eqnarray}
\int_0^T\rho(1-\phi_m)\varphi_t\rightarrow0,\ \ \int^T_0\int_{\mathbb{R}^3}\left[\rho(1-\phi_m)u\nabla\varphi-\rho u\varphi\nabla\phi_m\right]\rightarrow0.\label{dont}
\end{eqnarray}
So we conclude
\begin{eqnarray}
\int^T_0\int_{\mathbb{R}^3}\left(\rho\varphi_t+\rho u\nabla\varphi\right)=0.
\end{eqnarray}
Here we use Hardy's inequality and Lebesgue's theorem. Then we get
\begin{lema}\label{lemma7}
Let $\rho\in L^2((0,T)\times D)$, $u\in L^2(0,T;H^1_0(D))$ be a solution of (\ref{fequ1}) in $D'((0,T)\times D)$. Then prolonging $\rho,\ u$ to be zero on $\mathbb{R}^3\backslash D$, the equation holds in $D'((0,T)\times \mathbb{R}^3)$.
\end{lema}

We introduce $Riesz\ integral\ operator\ \mathcal{R}$ and $singular\ integral\ operator \mathcal{A}$.
\begin{eqnarray}
\mathcal{R}_i=(-\Delta^{-\frac{1}{2}})\partial_{x_i},\nonumber%\label{riesz1}
\end{eqnarray}
\begin{eqnarray}
\mathcal{A}_i=\Delta^{-1}\partial_{x_i},\nonumber%\label{riesz2}
\end{eqnarray}
And it holds
\begin{eqnarray}
\partial_{x_i}\mathcal{A}_j=-\mathcal{R}_i\mathcal{R}_j.\label{riesz3}
\end{eqnarray}
These operators have the following properties (see \cite{f2}(Lemma 5.5.2) or \cite{ca}).
\begin{lema}\label{lemma8}
The $Riesz\ operator\ \mathcal{R}_i$, $i=1,...,N$, defined in (\ref{riesz3}) is a bounded linear operator on $L^p(R^N)$ for any $1<p<\infty$.
\end{lema}
\begin{lema}\label{lemma9}
Let $v\in (L^1\cap L^2)(\mathbb{R}^3)$,
then $\mathcal{A}_i[v]\in (L^\infty\oplus L^2)(\mathbb{R}^3)$, and
\begin{align}
&\|\mathcal{A}_i[v]\|_{(L^\infty\oplus L^2)(\mathbb{R}^3)}\leq c\|v\|_{(L^1\cap L^2)(\mathbb{R}^3)},\nonumber\\
&\|\partial_{x_i}\mathcal{A}_i[v]\|_{L^p(\mathbb{R}^3)}\leq c(p)\|v\|_{L^p(\mathbb{R}^3)}\quad for\ any\ 1<p<\infty.\nonumber
\end{align}

\end{lema}

The next lemma is from \cite{f2}(Corollary 6.6.1).
\begin{lema}\label{lemma10}
Let $\{v_n\},\ \{w_n\}$ be the two sequences of vector functions,
\begin{eqnarray}
&v_n\rightharpoonup v\quad in\ L^p(D;\mathbb{R}^3),\quad w_n\rightharpoonup w \quad in\ L^q(D;\mathbb{R}^3),\nonumber\\
&B_n\rightharpoonup B\quad in\ L^p(D),\quad \frac{1}{p}+\frac{1}{q}\leq1,\quad1<p,\ q<\infty.\nonumber
\end{eqnarray}
Then the following terms satisfy in distributions:
\begin{align}
&v_n\cdot(\nabla\Delta^{-1} div)[w_n]\!-\!w_n\cdot(\nabla\Delta^{-1} div)[v_n]
\!\rightarrow\! v\cdot(\nabla\Delta^{-1} div)[w]\!-\!w\cdot(\nabla\Delta^{-1} div)[v],\nonumber\\
&v_n(\nabla\Delta^{-1}\nabla)[B_n]-B_n(\nabla\Delta^{-1} div)[v_n]
\rightarrow v(\nabla\Delta^{-1}\nabla)[B]-B(\nabla\Delta^{-1} div)[v].\nonumber
\end{align}
\end{lema}
\begin{lema}\label{lemma11}
Let $\eta\in C_0^{\infty}(D),\ \psi\in C_0^{\infty}(0,T)$ and $B$ is a bounded measurable function satisfying
\begin{eqnarray}
\partial_tB+div(Bu)=h\quad in\ D'((0,T)\times D),\ with\ h\in L^2(0,T;W^{-1,q}(D)).\nonumber
\end{eqnarray}

Let $\varphi(t,x)=\psi(t)\eta(x)\mathcal{A}[B(t,x)]$ be a test function of (\ref{fequ2}), providing $B,\ u$ prolonged to zero outside $D$. Then we have
\begin{align}
&\int^T_0\int_{D}\psi\eta(\overline{p}B-S:(\nabla\Delta^{-1}\nabla)[B])\nonumber\\
&=\int^T_0\int_{D}\left\{\psi(S\nabla\eta)A[B]-\psi\eta\rho u\otimes u:(\nabla\Delta^{-1}\nabla)[B]-\psi \overline{p}\nabla\eta A[B]\right\}\nonumber\\
&+\int^T_0\int_{D}\left\{-\psi[(\rho u\otimes u)\nabla\eta]A[B]-\psi\eta(\nabla d\otimes\nabla d):(\nabla\Delta^{-1}\nabla)[B]\right\}\nonumber\\
&+\int^T_0\int_{D}\left\{-\psi[(\nabla d\otimes\nabla d)\nabla\eta]\mathcal{A}[B]+\frac{1}{2}(|\nabla d|^2\!+\!F(d))\psi\eta B\right\}\nonumber\\
&+\!\int^T_0\!\!\int_{D}\!\left\{\frac{1}{2}\left(|\nabla d|^2\!+\!F(d)\right)\psi\nabla\eta \mathcal{A}[B]\!+\!\psi\eta(\mu_2N\otimes d+\mu_3d\otimes N)\!:\!(\nabla\Delta^{-1}\nabla)[B])\right\}\nonumber\\
&+\int^T_0\int_{D}\left\{\psi\left((\mu_2N\otimes d+\mu_3d\otimes N)\nabla\eta\right)A[B]-\psi_t\eta\rho u\mathcal{A}[B]\right\}\nonumber\\
&+\int^T_0\int_{D}\left\{-\psi\eta\rho u\mathcal{A}[h]+\psi\eta\rho u\mathcal{A}[div(Bu)]\right\},\nonumber
\end{align}
where $\ S=\mu_4A+\mu_7tr(A)I$.
\end{lema}
\begin{proof}
Form the definition of $\mathcal{A}$, we have
\begin{eqnarray}
&&div\varphi=\psi\nabla\eta\mathcal{A}[B]+\psi\eta B,\nonumber\\
&&\nabla\varphi=\psi\nabla\eta\mathcal{A}[B]+\psi\eta(\nabla\Delta^{-1}\nabla)[B],\nonumber\\
&&\varphi_t=\psi_t\eta\mathcal{A}[B]+\psi\eta\mathcal{A}[h]-\psi\eta\mathcal{A}[div(Bu)].\nonumber
\end{eqnarray}
Taking above relations into (\ref{fequ2}), we conclude this lemma.
\end{proof}
\begin{lema}\label{lemma12}
Let $(\rho_\epsilon,\ u_\epsilon,\ d_\epsilon)$ be a solution of (\ref{approx1})-(\ref{approb3}). And $(\rho,u,d)$ solves (\ref{fequ1})-(\ref{fequ3}). Then
\begin{eqnarray}
&&\int^T_0\int_{D}\psi\phi\left(a\rho^\gamma_\epsilon+\delta\rho^\beta_\epsilon-(\mu_4+\mu_7)divu_\epsilon\right)\rho_\epsilon \rightarrow\nonumber\\
&&\int^T_0\int_{D}\psi\phi\left(\overline{p}-(\mu_4+\mu_7)divu\right)\rho \quad as\ \epsilon\rightarrow0\nonumber
\end{eqnarray}
for any $\psi\in C_0^\infty(0,T),\ \phi\in C_0^\infty(D)$.
\end{lema}
 Before proving the lemma, we prolong $\rho_\epsilon$ to be zero outside $D$
\begin{eqnarray}
\partial_t\rho_\epsilon=\left\{\begin{array}{ll}
\epsilon\Delta\rho_\epsilon-div(\rho_\epsilon u_\epsilon)\ &x\in D\\
 0\ \quad \ &x\in \mathbb{R}^3\setminus D.
\end{array}\right.
\end{eqnarray}
Since $\rho,\ u$ vanish outsider $D$, we have $div(\rho u)=0,\ x\in \mathbb{R}^3\setminus D$ and
\begin{eqnarray}
div(1_{D}\nabla\rho_\epsilon)=\left\{\begin{array}{ll}
\Delta\rho_\epsilon\ &x\in D,\\
 0\ \quad \ &x\in \mathbb{R}^3\setminus D.
\end{array}\right.
\end{eqnarray}
\begin{proof}
Here $\rho_\epsilon,\ \rho$ are extended to zero outside of $D$.
Taking $\varphi=\psi\eta\mathcal{A}[\rho_\epsilon]$ as a test function and using {Lemma \ref{lemma11}}, we get
\begin{align}
&\int^T_0\int_{D}\left[\psi\eta\left(\left(a\rho^\gamma_\epsilon+\delta\rho^\beta_\epsilon\right)\rho_\epsilon-(\mu_4A_\epsilon+\mu_7tr(A_\epsilon)I\right):(\nabla\Delta^{-1}\nabla)[\rho_\epsilon])\right]\nonumber\\
&=\int^T_0\!\!\int_{D}\!\left\{\psi(S_\epsilon\nabla\eta)A[\rho_\epsilon]-\psi\eta\rho_\epsilon u_\epsilon\otimes u_\epsilon:(\nabla\Delta^{-1}\nabla)[\rho_\epsilon]-\psi(a\rho^\gamma_\epsilon+\delta\rho^\beta_\epsilon)\nabla\eta A[\varrho_\epsilon]\right\}\nonumber\\
&+\int^T_0\int_{D}\left\{-\psi[(\rho_\epsilon u_\epsilon\otimes u_\epsilon)\nabla\eta]A[\rho_\epsilon]-\psi\eta(\nabla d_\epsilon\otimes\nabla d_\epsilon):(\nabla\Delta^{-1}\nabla)[\rho_\epsilon]\right\}\nonumber\\
&+\int^T_0\int_{D}\left\{-\psi[(\nabla d_\epsilon\otimes\nabla d_\epsilon)\nabla\eta]\mathcal{A}[\rho_\epsilon]+\frac{1}{2}(|\nabla d_\epsilon|^2+F(d_\epsilon))\psi\eta\rho_\epsilon\right\}\nonumber\\
&+\int^T_0\!\!\int_{D}\!\left\{\frac{1}{2}(|\nabla d_\epsilon|^2\!+\!F(d_\epsilon))\psi\nabla\eta \mathcal{A}[\rho_\epsilon]+\psi\eta(\mu_4A_\epsilon+\mu_7tr(A_\epsilon)I):(\nabla\Delta^{-1}\nabla)[\rho_\epsilon]\right\}\nonumber\\
&+\int^T_0\int_{D}\left\{\psi(\mu_4A_\epsilon+\mu_7tr(A_\epsilon)I)\nabla\eta)A[\rho_\epsilon]-\psi_t\eta\rho_\epsilon u\mathcal{A}[\rho_\epsilon]+\psi\eta\rho_\epsilon u_\epsilon\mathcal{A}[div(\rho_\epsilon u_\epsilon)]\right\}\nonumber\\
&+\int^T_0\int_{D}\left\{-\psi\eta\rho_\epsilon u_\epsilon\mathcal{A}[\epsilon div(1_{D}\nabla\rho_\epsilon)]+\epsilon\psi\eta(\nabla u_\epsilon\nabla\rho_\epsilon)A[\rho_\epsilon]\right\}\nonumber\\
&=\sum^{14}_{i=1}I^\epsilon_i.\nonumber
\end{align}
Similarly, taking $\varphi=\psi\eta\mathcal{A}[\rho]$ as a test function of (\ref{fequ2}), we have
\begin{align}
&\int^T_0\int_{D}\psi\eta\left\{\overline{p}\rho-(\mu_4A+\mu_7tr(A)I):(\nabla\Delta^{-1}\nabla)[\rho]\right\}\nonumber\\
&=\int^T_0\int_{D}\left\{\psi(S\nabla\eta)A[\rho]-\psi\eta\rho u\otimes u:(\nabla\Delta^{-1}\nabla)[\rho]-\psi\overline{p}\nabla\eta A[\varrho]\right\}\nonumber\\
&+\int^T_0\int_{D}\left\{-\psi[(\rho u\otimes u)\nabla\eta]A[\rho]-\psi\eta(\nabla d\otimes\nabla d):(\nabla\Delta^{-1}\nabla)[\rho]\right\}\nonumber\\
&+\int^T_0\int_{D}\left\{-\psi[(\nabla d\otimes\nabla d)\nabla\eta]\mathcal{A}[\rho]+\frac{1}{2}(|\nabla d|^2+F(d))\psi\eta\rho\right\}\nonumber\\
&+\int^T_0\int_{D}\left\{\frac{1}{2}(|\nabla d|^2+F(d))\psi\nabla\eta \mathcal{A}[\rho]+\psi\eta(\mu_4A+\mu_7tr(A)I):(\nabla\Delta^{-1}\nabla)[\rho])\right\}\nonumber\\
&+\int^T_0\int_{D}\left\{\psi(\mu_4A+\mu_7tr(A)I)\nabla\eta)A[\rho]-\psi_t\eta\rho u\mathcal{A}[\rho]+\psi\eta\rho u\mathcal{A}[div(\rho u)]\right\}\nonumber\\
&=\sum^{12}_{i=1}I_i.\nonumber
\end{align}
By {Proposition \ref{prop4}}, we have
\begin{align}
&I_{13}=\int^T_0\int_{D}-\psi\eta\rho_\epsilon u_\epsilon\mathcal{A}[\epsilon div(1_{D}\nabla\rho_\epsilon)]\nonumber\\
&\leq\sqrt{\epsilon}\int^T_0\psi\sqrt{\epsilon}\|\nabla\rho_\epsilon\|_{L^2(D)}
\|\rho_\epsilon\|_{L^3(D)}\|u_\epsilon\|_{L^6(D)}dt\rightarrow0\quad as\ \epsilon\rightarrow0.\label{index1}
\end{align}
and
\begin{align}
&I_{14}=\int^T_0\int_{D}\epsilon\psi\eta(\nabla u_\epsilon\nabla\rho_\epsilon)A[\rho_\epsilon]\nonumber\\
&\leq\sqrt{\epsilon}\int^T_0\psi\sqrt{\epsilon}\|\nabla\rho_\epsilon\|_{L^2(D)}
\|A[\rho_\epsilon]\|_{L^\infty(D)}\|\nabla u_\epsilon\|_{L^2(D)}dt\rightarrow0\quad as\ \epsilon\rightarrow0.\label{index2}
\end{align}
Next we show that it still holds $I^\epsilon_i\rightarrow I_i,\ i=1,2\cdots,12$.
Using
\begin{eqnarray}
\rho_\epsilon\rightarrow\rho\quad in\ C([0,T];L^\beta_{weak}(D)),\nonumber
\end{eqnarray}
we have
\begin{align}
&A[\rho_\epsilon]\rightarrow A[\rho]\qquad in\ C([0,T]\times D),\nonumber\\
&R_iR_j[\rho_\epsilon]\rightarrow R_iR_j[\rho]\qquad in\ C([0,T];L^\beta_{weak}(D)).\nonumber
\end{align}
And it is easy to get $I^\epsilon_1\rightarrow I_1,\ I^\epsilon_3\rightarrow I_3,\ I^\epsilon_4\rightarrow I_4,\ I^\epsilon_{10}\rightarrow I_{10},\ I^\epsilon_{11}\rightarrow I_{11}$. By virtue of $d_\epsilon\rightarrow d\ in\ W^{1,k}(D),\ 0\leq k<6$, we have
$I^\epsilon_5\rightarrow I_5,\ I^\epsilon_6\rightarrow I_6,\ I^\epsilon_7\rightarrow I_7,\ I^\epsilon_8\rightarrow I_8$. Noting $\mu_2=-\mu_3$, we have $I^\epsilon_9=I_9=0$.
It only leaves us to consider $I^\epsilon_2,\ I^\epsilon_{12}$.
\begin{align}
I^\epsilon_2+I^\epsilon_{12}&=\int^T_0\int_{D}\left\{\psi\eta\rho_\epsilon u_\epsilon\mathcal{A}[div(\rho_\epsilon u_\epsilon)]-\psi\eta\rho_\epsilon u_\epsilon\otimes u_\epsilon:(\nabla\Delta^{-1}\nabla)[\rho_\epsilon]\right\}\nonumber\\
&=\int^T_0\int_{D}\left\{\psi u_\epsilon(\rho_\epsilon\nabla\triangle^{-1}div(\eta\rho_\epsilon u_\epsilon)-\eta\rho_\epsilon(\nabla\Delta^{-1}\nabla)[\rho_\epsilon]u_\epsilon\right\}\nonumber\\
&\rightarrow\int^T_0\int_{D}\left\{\psi u(\rho\nabla\triangle^{-1}div(\eta\rho u)-\eta\rho(\nabla\Delta^{-1}\nabla)[\rho]u\right\}\nonumber\\
&=\int^T_0\int_{D}\left\{\psi\eta\rho u\mathcal{A}[div(\rho u)]-\psi\eta\rho u\otimes u:(\nabla\Delta^{-1}\nabla)[\rho]\right\}\nonumber\\
&=I_2+I_{12}.\nonumber
\end{align}
Here we used {Lemma \ref{lemma10}}. It has been deduced that
\begin{eqnarray}
\int^T_0\int_{D}\psi\eta\left((a\rho^\gamma_\epsilon+\delta\rho^\beta_\epsilon)\rho_\epsilon-(\mu_4A_\epsilon+\mu_7tr(A_\epsilon)I):(\nabla\Delta^{-1}\nabla)[\rho_\epsilon]\right)\nonumber\\
\rightarrow\int^T_0\int_{D}\psi\phi\left(\overline{p}\rho-(\mu_4A+\mu_7tr(A)I):(\nabla\Delta^{-1}\nabla)[\rho]\right).\label{sum1}
\end{eqnarray}
Straightforward computation yields
\begin{align}
&\int^T_0\int_{D}\psi\eta(\mu_4A_\epsilon+\mu_7tr(A_\epsilon)I):(\nabla\Delta^{-1}\nabla)[\rho_\epsilon]\nonumber\\
&=\int^T_0\!\!\int_{D}\!\left\{\psi\rho_\epsilon(\mu_4+\mu_7)div(\eta u_\epsilon)\!-\!\psi\rho_\epsilon\left[\mu_4(\nabla\Delta^{-1}\nabla):(u_\epsilon\otimes\nabla\eta)+\mu_7 \rho_\epsilon u_\epsilon\cdot\nabla\eta\psi\right]\right\}\label{sum2}
\end{align}
and
\begin{align}
&\int^T_0\int_{D}\psi\eta(\mu_4A+\mu_7tr(A)I):(\nabla\Delta^{-1}\nabla)[\rho]\nonumber\\
&=\int^T_0\int_{D}\left\{\psi\rho(\mu_4+\mu_7)div(\eta u)-\psi\rho\left[\mu_4(\nabla\Delta^{-1}\nabla):(u\otimes\nabla\eta)+\mu_7\rho u\cdot\nabla\psi\eta\right]\right\},\label{sum3}
\end{align}
Also we have
%\begin{eqnarray}
%\int^T_0\int_{R^N}\psi\rho_\epsilon(\mu_4+\mu_7)div(\phi u_\epsilon)
%\rightarrow\int^T_0\int_{R^N}\psi\rho(\mu_4+\mu_7)div(\phi u),\nonumber
%\end{eqnarray}
\begin{eqnarray}
\int^T_0\int_{D}\psi\rho_\epsilon\left[\mu_4(\nabla\Delta^{-1}\nabla):(u_\epsilon\otimes\nabla\eta)+\mu_7 \rho_\epsilon u_\epsilon\cdot\nabla\eta\psi\right]\nonumber\\
\rightarrow\int^T_0\int_{D}\psi\rho\left[\mu_4(\nabla\Delta^{-1}\nabla):(u\otimes\nabla\eta)+\mu_7\rho u\cdot\nabla\psi\eta\right].\label{sum4}
\end{eqnarray}
(\ref{sum1})-(\ref{sum4}) lead to our conclusion.
\end{proof}
%\subsection{Strong Convergence of The Density}
We need the following lemma about renormalized solution(see \cite{f2}). For the sake of comleteness, we rewrite the proof.
\begin{lema}\label{k1}
Assume $\rho\in L^2((0,T)\times D),\ u\in L^2(0,T;H^1_0(D))$ solves (\ref{fequ1}) in the sense of $D'((0,T)\times \mathbb{R}^3)$. Then $(\rho,u)$ is a renormalized solution of (\ref{fequ1}).
\end{lema}
\begin{proof}
Taking the mollified operator on both side of (\ref{fequ1}), we have
\begin{align}
&\partial_t[\rho]^\epsilon_x+\nabla[\rho]^\epsilon_x u+[\rho]^\epsilon_xdivu
=[div(\rho u)]^\epsilon_x-div([\rho]^\epsilon_x u).\label{rre1}
\end{align}
Then we get
\begin{align}
&\int_0^T\int_{D}\left|[div(\rho u)]^\epsilon_x-div([\rho]^\epsilon_x u)\right|\nonumber\\
&\leq\int_0^T\int_{D}\left\{\left|\int_{R^N}\rho(y)(u(x)-u(y))\nabla\theta^\epsilon(|x-y|)dy\right|+|[\rho]^\epsilon_xdivu|\right\}.\label{rre2}
\end{align}
Noticing
\begin{align}
&\int_0^T\int_{D}\left|\int_{R^N}\rho(y)(u(x)-u(y))\nabla\theta^\epsilon(|x-y|)dy\right|\nonumber\\
&=\int_0^T\int_{D}\left|\int_{R^N}\rho(x-z)\frac{u(x)-u(x-z)}{|z|}\nabla\theta^\epsilon(|z|)|z|dz\right|.
\end{align}
and the second term of the right side of (\ref{rre2}) is bounded, we have
\begin{eqnarray}
[div(\rho u)]^\epsilon_x-div([\rho]^\epsilon_x u)\rightarrow0\ in\ L^1((0,T)\times D),\quad as\ \epsilon\rightarrow0.\nonumber
\end{eqnarray}
by Lebesgue's theorem.
Multiplying $B'([\rho]^\epsilon_x)$ on (\ref{rre1}), the equation reads
\begin{align}
\partial_t(B[\rho]^\epsilon_x)\!+\!\nabla(B[\rho]^\epsilon_x)u\!+\!B([\rho]^\epsilon_x)divu\!+\!B'([\rho]^\epsilon_x)[\rho]^\epsilon_xdivu\!-\!B([\rho]^\epsilon_x)divu
\!=\!B'([\rho]^\epsilon_x)s^\epsilon.\nonumber
\end{align}
Passing to the limit for $\epsilon\rightarrow0$, we conclude
\begin{eqnarray}
(B(\rho))_t+div(B(\rho)u)+(B'(\rho)\rho-B(\rho))divu=0.\label{userenor}
\end{eqnarray}
\end{proof}
Taking $B(z)=z\log z$ in (\ref{userenor}), we obtain
\begin{eqnarray}
\int^T_0\int_{D}\rho divu=\int_{D}\rho_0\log \rho_0dx-\int_{D}\rho(T)\log \rho(T)dx.\nonumber
\end{eqnarray}
Similar to lemma \ref{k1}, it holds
\begin{eqnarray}
B_t(\rho_\epsilon)\!+\!div(B(\rho_\epsilon)u_\epsilon)\!+\!(B'(\rho_\epsilon)\rho_\epsilon\!-\!B(\rho_\epsilon))divu_\epsilon\!-\!\epsilon\Delta B(\rho_\epsilon)\!=\!-B''(\rho_\epsilon)|\nabla\rho_\epsilon|^2\!\leq\!0.\nonumber
\end{eqnarray}
The term $\nabla\rho_\epsilon\cdot n=0$ leads to $\int_{D}\Delta\left(B(\rho_\epsilon)\right)dx=0$. Then we have
%As $B(\rho)=\rho\log\rho$ is convex and globally lipschitz on $R^+$,
\begin{eqnarray}
\int^T_0\int_{D}\rho_\epsilon divu_\epsilon \leq\int_{D}\rho_0\log \rho_0dx-\int_{D}\rho_\epsilon(T)\log \rho_\epsilon(T)dx.\nonumber
\end{eqnarray}
Let $\psi_m\in C_0^{\infty}(0,T),\ \eta_m\in C_0^\infty(D)$, satisfying $\psi_m\rightarrow1$ and $\eta_m\rightarrow1$. By virtue of {Lemma \ref{lemma12}}, we have
\begin{align}
&\limsup_{\epsilon\rightarrow0^+}\int^T_0\int_{D}\psi_m\eta_m(a\rho^\gamma_\epsilon+\delta\rho^\beta_\epsilon)\rho_\epsilon \nonumber\\
&=\limsup_{\epsilon\rightarrow0^+}\int^T_0\!\!\int_{D}\left[\psi_m\eta_m(a\rho^\gamma_\epsilon+\delta\rho^\beta_\epsilon
\!-\!(\mu_4+\mu_7)divu_\epsilon)\rho_\epsilon+\psi_m\eta_m(\mu_4\!+\!\mu_7)\rho_\epsilon divu_\epsilon\right]\nonumber\\
&\leq\int^T_0\int_{D}\psi_m\eta_m\left(\overline{p}-(\mu_4+\mu_7)divu\right)\rho
+(\mu_4+\mu_7)\limsup_{\epsilon\rightarrow0^+}\int^T_0\int_{D}\rho_\epsilon divu_\epsilon
\nonumber\\
&\quad+(\mu_4+\mu_7)\limsup_{\epsilon\rightarrow0^+}\int^T_0\int_{D}\rho_\epsilon|1-\psi_m\eta_m||divu_\epsilon|\nonumber\\
&\leq\int^T_0\int_{D}\psi_m\eta_m(\overline{p}\rho
+2(\mu_4+\mu_7)\limsup_{\epsilon\rightarrow0^+}\int^T_0\int_{D}\rho_\epsilon|1-\psi_m\eta_m||divu_\epsilon|\nonumber\\
&\quad+(\mu_4+\mu_7)\int_{D}\rho(T)\log \rho(T)dx-\liminf_{\epsilon\rightarrow0^+}\int_{D}\rho_\epsilon(T)\log \rho_\epsilon(T)dx\nonumber\\
&\leq\int_{D}\int_{D}\psi_m\eta_m \overline{p}\rho +o(m^{-1}),\label{last}
\end{align}
The last inequality is due to the fact: $B(z)$ is convex and globally lipschitz on $R^+$. Thus we have proved
\begin{eqnarray}
\lim_{\epsilon\rightarrow0^+}\sup\int^T_0\int_{D}\psi_m\eta_m(a\rho^\gamma_\epsilon+\delta\rho^\beta_\epsilon)\rho_\epsilon \leq\int^T_0\int_{D}\psi_m\eta_m\overline{p}\rho , \quad  m\gg1.\nonumber
\end{eqnarray}
Setting $p(z)=az^\gamma+\delta z^\beta$, it holds
\begin{align}
&\int^T_0\int_{D}\psi_m\eta_m(p(\rho_\epsilon)-p(v))(\rho_\epsilon-v)\nonumber\\
&=\int^T_0\int_{D}\psi_m\eta_m\left(p(\rho_\epsilon)\rho_\epsilon-p(\rho_\epsilon)v-p(v)\rho_\epsilon+p(v)v\right)\geq0.\nonumber
\end{align}
Then
%\begin{eqnarray}
%\int^T_0\int_{D} \overline{p}\rho+\lim_{\epsilon\rightarrow0}\int^T_0\int_{D}\psi_m\eta_m(-p(\rho_\epsilon)v-p(v)\rho_\epsilon+p(v)v)\geq0,\nonumber
%\end{eqnarray}
\begin{eqnarray}
\int^T_0\int_{D}\psi_m\eta_m \overline{p}\rho dx dt+\int^T_0\int_{D}\psi_m\eta_m(-\overline{p}v-p(v)\rho+p(v)v)\geq0.\nonumber
\end{eqnarray}
Let $m\rightarrow\infty$, we obtain
\begin{eqnarray}
\int^T_0\int_{D}(\overline{p}-p(v))(\rho-v)\geq0.\nonumber
\end{eqnarray}
Choosing $v=\rho+\zeta\varphi$, for any $\varphi$, then $\zeta\rightarrow0$ yields
\begin{eqnarray}
\overline{p}=a\rho^\gamma+\delta\rho^\beta.
\end{eqnarray}
We can rewrite the system (\ref{fequ1})-(\ref{fequ3}) as
\begin{align}
&\rho_t+div(\rho u)=0,\label{sequ1}\\
&(\rho u)_t+\mathrm{div}(\rho u\otimes u)+\nabla(p-\frac{1}{2}|\nabla d|^2-F(d))+\nabla\cdot(\nabla d\odot\nabla d)=\mathrm{div}\widehat{\sigma},\label{sequ2}\\
&\lambda_1d_t+\lambda_1u\cdot\nabla d-\lambda_1\Omega d+\Delta d-f(d)=0,\label{sequ3}
\end{align}
where $\widehat{\sigma}=\mu_3(N\otimes d-d\otimes N)+\mu_4A+\mu_7tr(A)I,\ \ p=a\rho^\gamma+\delta\rho^\beta$.
\begin{prop}\label{se1}
Let $\Omega$ be a bounded domain in class of $C^{2+\nu}$, and $\beta>\{4,\frac{6\gamma}{2\gamma-3}\}$. Then there exists a finite energy weak solution $(\rho,u,d)$ to the problem (\ref{sequ1})-(\ref{sequ3}), (\ref{approini1})-(\ref{approb3}). And $\rho,\ u,\ d$ satisfy the following estimates:
\begin{align}
&\|\sqrt{\rho}u\|_{L^\infty(0,T;L^2(D))}\leq CE(\rho_0,q,d_0),\quad\|\rho\|_{L^\infty(0,T;L^\gamma(D))}\leq CE(\rho_0,q,d_0),\nonumber\\
&\delta\|\rho\|^\beta_{L^\infty(0,T;L^\beta(D))}\leq CE(\rho_0,q,d_0),\quad\|u\|_{L^2(0,T;H^1_0(D))}\leq CE(\rho_0,q,d_0),\nonumber\\
&\|d\|_{L^\infty(0,T;H^1(D))}\leq CE(\rho_0,q,d_0),\quad\|d\|_{L^2(0,T;H^2(D))}\leq CE(\rho_0,q,d_0),\nonumber\\
&\|N\|_{L^2((0,T)\times D)}\leq CE(\rho_0,q,d_0),\quad\|d_t\|_{L^2(0,T;L^{\frac{3}{2}}(D))}\leq CE(\rho_0,q,d_0).\nonumber
\end{align}
\end{prop}

\section{Let $\delta\rightarrow\infty$}\label{4}

\setcounter{equation}{0}

We will find $\rho_{0,\delta},\ q_{0,\delta}\ and\ d_{0,\delta}$ which satisfy all the requisite of this paper(for example (\ref{approini1})-(\ref{approb3}), etc).
Firstly, one can find $\rho_{0,\delta}\in C^{2+\alpha}(\overline{D})$, satisfying
\begin{eqnarray}
0<\delta\leq\rho_{0,\delta}\leq\delta^{-\frac{1}{\beta}},\quad\nabla\rho_{0,\delta}\cdot n|_{\partial D}=0\ \mathrm{and}\
\int_{D}|\rho_{0,\delta}-\rho_0|^\gamma dx\leq\delta^\gamma.\nonumber
\end{eqnarray}
Indeed one can easily find $\rho_\delta\in C_0^\infty(D)$, $|\rho_\delta-\rho_0|_{L^\gamma}<\delta$, with $0<\rho_\delta<\delta^{-\frac{1}{\beta}}$. Let
$\rho_{0,\delta}=\rho_\delta+\delta$, with $\nabla\rho_{0,\delta}\cdot n|_{\partial D}=0$. Secondly, Let
\begin{eqnarray}
\overline{q}_\delta=
\left\{\begin{array}{ll}
q(x)\sqrt{\frac{\rho_{0,\delta}}{\rho_0}},\quad &\rho_0>0,\\
0          &\rho_0=0.
\end{array}\right.\label{la1}
\end{eqnarray}
(\ref{la1}) leads to
\begin{eqnarray}
\frac{|\overline{q}_\delta|^2}{\rho_{0,\delta}}\ \mathrm{is\ bounded\ in}\ L^1(D),\ \mathrm{independent\ of}\ \delta.\nonumber
\end{eqnarray}
Then we take $\tilde{q}_\delta\in C^2(\overline{D})$ such that
\begin{eqnarray}
\| \frac{\overline{q}_\delta}{\sqrt{\rho_{0,\delta}}}-\tilde{q}_\delta\|_{L^2(D)}<\delta.\nonumber
\end{eqnarray}
Let $q_{0,\delta}=\tilde{q}_\delta\sqrt{\rho_{0,\delta}}$
\begin{eqnarray}
&\frac{|q_{0,\delta}|^2}{\rho_{0,\delta}}\ \mathrm{is\ bounded\ in}\ L^1(D)\ \mathrm{independently\ of}\ \delta,\nonumber\\
&q_{0,\delta}\rightarrow q\quad \mathrm{in}\ L^{\frac{2\gamma}{\gamma+1}}(D)\quad \mathrm{as}\ \delta\rightarrow0.\nonumber
\end{eqnarray}
Indeed
\begin{align}
&(\!\!\!\int_{\{\rho_0(x)>0\}}|q_{0,\delta}-q|^{\frac{2\gamma}{\gamma+1}}dx)^{\frac{\gamma+1}{2\gamma}}=(\int_{\{\rho_0(x)>0\}}
|\tilde{q}_\delta\sqrt{\rho_{0,\delta}}-\frac{\sqrt{\rho_0}}{\sqrt{\rho_{0,\delta}}}\overline{q}_\delta|^{\frac{2\gamma}{\gamma+1}}dx)^{\frac{\gamma+1}{2\gamma}}\nonumber\\
\!\!\!&=(\!\!\int_{\{\rho_0(x)>0\}}\!\!
|\sqrt{\rho_{0,\delta}}(\tilde{q}_\delta-\frac{\overline{q}_\delta}{\sqrt{\rho_{0,\delta}}})|^{\frac{2\gamma}{\gamma+1}})^{\frac{\gamma+1}{2\gamma}}\!\!\!
+(\!\!\int_{\{\rho_0(x)>0\}}\!\!
|\frac{\overline{q}_\delta}{\sqrt{\rho_{0,\delta}}}(\sqrt{\rho_0}-\!\sqrt{\rho_{0,\delta}})|^{\frac{2\gamma}{\gamma+1}})^{\frac{\gamma+1}{2\gamma}}\nonumber\\
\!\!\!&=I_1+I_2.
\end{align}
It is easy to get
\begin{align}
&I_1\rightarrow0,\quad \delta\rightarrow0,\nonumber\\
&I_2<\|\frac{\overline{q}_\delta}{\sqrt{\rho_{0,\delta}}}\|_{L^2}(\|\rho_0\|^{\frac{1}{2}}_{L^\gamma}
-\|\rho_{0,\delta}\|^{\frac{1}{2}}_{L^\gamma})<C\delta\rightarrow0\nonumber
\end{align}
and
\begin{align}
(\int_{\{\rho_0(x)=0\}}|q_{0,\delta}|^{\frac{2\gamma}{\gamma+1}}dx)^{\frac{\gamma+1}{2\gamma}}
&=(\int_{\{\rho_0(x)=0\}}|\tilde{q}_\delta\sqrt{\rho_{0,\delta}}|^{\frac{2\gamma}{\gamma+1}}dx)^{\frac{\gamma+1}{2\gamma}}\nonumber\\
&\leq\|\tilde{q}_\delta\|_{L^2(D)}\|\rho_{0,\delta}-\rho_0\|_{L^\gamma(D)}\rightarrow0\ as\ \delta\rightarrow0.\nonumber
\end{align}
Thirdly, we can easily find $\|d_{0,\delta}-d_0\|_{H^2(D)}<\delta$, $|d_{0,\delta}|=1$.
Let $(\rho_\delta,\ u_\delta,\ d_\delta)$ be the approximate solution of the problem (\ref{sequ1})-(\ref{sequ3}) with the initial data $(\rho_{0,\delta},\ d_{0,\delta},\ q_{0,\delta})$ of (\ref{approini1})-(\ref{approb3}). Here we still use $(\rho,u,d)$ instead of $(\rho_\delta,u_\delta,d_\delta)$ for convenience.
Noting
\begin{eqnarray}
\rho\in L^\infty(0,T;L^\gamma(D))\cap L^\infty(0,T;L^\beta(D)),\quad u\in L^2(0,T;H^1_0(D)),\nonumber
\end{eqnarray}
with some $\beta\geq2$, we get that $(\rho,u)$ is a renormalized solution of (\ref{sequ1}) by {Lemma \ref{k1}}.
Similarly as {Lemma \ref{lemma6}}, Let $\varphi(t,x)=\psi(t)B[b(\rho)-m_0]$, $b(\rho)=\rho^\theta,\ m_0=\oint_{D} b(\rho)$, and put $\theta>0$ determined later.
\begin{align}
&\int^T_0\int_{D}\psi\left(a\rho^{\gamma+\theta}+\delta\rho^{\beta+\theta}\right)\nonumber\\
&=\int^T_0\int_{D}\left\{\psi m_0(a\rho^\gamma+\delta\rho^\beta)-\psi_t\rho u B[b(\rho)-m_0]+\psi\rho uB[div(b(\rho)u)]\right\}\nonumber\\
&+\int^T_0\int_{D}\left\{\psi\rho uB[(b'(\rho)\rho-b(\rho))divu]-\psi\rho u\otimes u:\nabla B[b(\rho)-m_0]\right\}\nonumber\\
&+\int^T_0\int_{D}\left\{\psi\widehat{\sigma}:\nabla B[b(\rho)-m_0]-\psi(\nabla d\otimes\nabla d):\nabla B[b(\rho)-m_0]\right\}\nonumber\\
&+\int^T_0\int_{D}\frac{1}{2}\left(|\nabla d|^2+F(d)\right)(b(\rho)-m_0)\nonumber\\
&=\sum^8_{i=1}I_i.\nonumber
\end{align}
In view of {Proposition \ref{se1}}, we estimate $I_i$ as follows:
\begin{align}
I_1=\int^T_0\int_{D}\psi m_0\left(a\rho^\gamma+\delta\rho^\beta\right)\leq C(E_0),\quad \theta\leq\gamma;\nonumber
\end{align}
\begin{align}
I_2&=\int^T_0\int_{D}-\psi_t\rho u B[b(\rho)-m_0]\nonumber\\
&\leq C(E_0)\|\sqrt{\rho}\|_{L^\infty(0,T;L^{2\gamma}(D))}\|\sqrt{\rho}u\|_{L^\infty(0,T;L^2(D))}\times\nonumber\\
&\qquad\quad\|B[b(\rho)-m_0]\|_{L^\infty(0,T;L^p(D))}\int^T_0|\psi_t|dt\nonumber\\
&\leq C(E_0)\int^T_0|\psi_t|dt, \ \ p>\frac{2\gamma}{\gamma-1},\ \theta\leq\frac{5\gamma}{6}-\frac{1}{2};\nonumber
\end{align}
\begin{eqnarray}
I_3=\int^T_0\int_{D}\psi\rho u B[div(b(\rho)u)]\leq C(E_0),\ \theta\leq\frac{2\gamma}{3}-1;\nonumber
\end{eqnarray}
%like $I_3$, we have
\begin{align}
I_4=\int^T_0\int_{D}\psi\rho uB[(b'(\rho)\rho-b(\rho))divu]
\leq C(E_0),\quad \theta\leq\frac{2\gamma}{3}-1;\nonumber
\end{align}
\begin{eqnarray}
&&I_5=\int^T_0\int_{D}-\psi\rho u\otimes u:\nabla B[b(\rho)-m_0]\nonumber\\
&&\quad\leq\|u\|^2_{L^2(0,T;L^6(D))}\|\rho\|^{1+\theta}_{L^\infty(0,T;L^{\frac{3}{2}(1+\theta)})}\leq C(E_0),\quad \theta\leq\frac{2\gamma}{3}-1;\nonumber
\end{eqnarray}
\begin{eqnarray}
I_6=\int^T_0\int_{D}\frac{1}{2}\left(|\nabla d|^2+F(d)\right)(b(\rho)-m_0)\leq C(E_0);\nonumber
\end{eqnarray}
\begin{eqnarray}
I_7=\int^T_0\int_{D}-\psi(\nabla d\otimes\nabla d):\nabla B[b(\rho)-m_0]\leq C(E_0);\nonumber
\end{eqnarray}
and
\begin{eqnarray}
I_8=\int^T_0\int_{D} \psi\widehat{\sigma}:\nabla B[b(\rho)-m_0]\leq C(E_0),\ \theta\leq\frac{\gamma}{2}.\nonumber
\end{eqnarray}
Thus we have proved the following lemma:
\begin{lema}\label{la2}
Let $(\rho,u,d)$ be a solution to the problem (\ref{sequ1})-(\ref{sequ3}), then there exists a constant $C=C(\rho_{0,\delta},p_{0,\delta},d_{0,\delta})$ which is independent of $\delta$, such that
\begin{eqnarray}
\int^T_0\int_{D}\psi\left(a\rho^{\gamma+\theta}+\delta\rho^{\beta+\theta}\right)\leq C,\ \ \theta\leq\min\left\{1,\frac{2\gamma}{3}-1,\frac{\gamma}{2}\right\}.\label{lemdd}
\end{eqnarray}
\end{lema}
\bigskip

By virtue of {Proposition \ref{se1}} and (\ref{lemdd}), we have
\begin{eqnarray}
\begin{array}{ll}
\rho_\delta\rightarrow\rho\quad in\ C([0,T];L^\gamma_{weak}(D)),&u_\delta\rightharpoonup u\quad in\ L^2(0,T;H^1_0(D)),\\
\rho_\delta u_\delta\rightarrow\rho u\quad in\ C([0,T];L^{\frac{2\gamma}{\gamma+1}}_{weak}(D)),&\rho_\delta u_\delta\otimes u_\delta\rightarrow\rho u\otimes u\quad in\ D'((0,T)\times D),\\
d_\delta\rightharpoonup d\quad in\ L^2(0,T;H^2(D)),&d_\delta\rightharpoonup^* d\quad in\ L^\infty(0,T;H^1(D)),\\
d_{\delta t}\rightharpoonup d_t\quad in\ L^2(0,T;L^{\frac{3}{2}}(D)),&\rho^\gamma_\delta\rightharpoonup\overline{\rho^\gamma}\quad in\ L^{\frac{\gamma+\theta}{\gamma}}((0,T)\times D),\\
\delta\rho^\beta\rightarrow0\quad in\ L^{\frac{\beta+\theta}{\beta}}((0,T)\times D).&
\end{array}\nonumber
\end{eqnarray}
Passing to the limit for $\delta\rightarrow0$, $(\rho,u,d)$ satisfies
\begin{align}
&\rho_t+div(\rho u)=0,\label{tequ1}\\
&(\rho u)_t+\mathrm{div}(\rho u\otimes u)+\nabla(\overline{p}-\frac{1}{2}|\nabla d|^2-F(d))+\nabla\cdot(\nabla d\odot\nabla d)=\mathrm{div}\widehat{\sigma},\label{tequ2}\\
&\lambda_1d_t+\lambda_1u\cdot\nabla d-\lambda_1\Omega d+\Delta d-f(d)=0,\label{tequ3}
\end{align}
where $\widehat{\sigma}=\mu_3(N\otimes d-d\otimes N)+\mu_4A+\mu_7tr(A)I,\ \ \overline{p}=a\overline{\rho^\gamma}$.
At last, we want to show $(\rho,u)$ is a renormalized solution and $\overline{p}=p=a\rho^\gamma$. Here (\ref{tequ1}) still holds in $D'((0,T)\times \mathbb{R}^3)$ provided $\rho,u$ is prolonged to be zero outside $D$.
Let us define
\begin{eqnarray}
T_k=kT(\frac{z}{k}),\qquad
T(z)=
\left\{\begin{array}{ll}
z\quad &0\leq z\leq1,\\
concave          &1<z<2,\\
2\quad &z\geq3,\\
-T(z)&z<0.
\end{array}\right.\label{tk}
\end{eqnarray}
Let $\overline{\Upsilon}$ denote the weak convergence limit of $\Upsilon$ in this paper. We have the following lemma.
\begin{lema}\label{la3}
Let $(\rho_\delta,u_\delta,d_\delta)$ be a sequence satisfying (\ref{sequ1})-(\ref{sequ3}), and $(\rho,u,d)$ solve (\ref{tequ1})-(\ref{tequ3}), then
\begin{eqnarray}
&&\int^T_0\int_{D}\psi\phi(a\rho^\gamma_\delta-(\mu_4+\mu_7)divu_\delta)T_k(\rho_\delta)\rightarrow\nonumber\\
&&\int^T_0\int_{D}\psi\phi(a\overline{\rho^\gamma}-(\mu_4+\mu_7)divu)\overline{T_k(\rho)}\qquad as\ \delta\rightarrow0.\label{b1}
\end{eqnarray}
holds for any constant $k$.
\end{lema}
\begin{proof}
By virtue of Lemma \ref{lemma7}, (\ref{sequ1}) still holds. Taking $\varphi(t,x)=\psi(t)\eta(x)A\break [T_k(\rho_\delta)]$, $\psi\in D(0,T),\eta\in C_0^\infty(D)$ as a test function for (\ref{sequ2}), with a straightforward computation, we get
\begin{align}
&\int^T_0\int_{D}\psi\eta\left\{\left(a\rho^\gamma_\delta+\delta\rho^\beta_\delta\right)T_k(\rho_\delta)-(\mu_4A_\delta+\mu_7tr(A_\delta)I):(\nabla\Delta^{-1}\nabla)[T_k(\rho_\delta)]\right\}
\nonumber\\
&=\int^T_0\int_{D}\left\{\psi(S_\delta\nabla\eta)A[T_k(\rho_\delta)]-\psi\eta\rho_\delta u_\delta\otimes u_\delta:(\nabla\Delta\nabla)[T_k(\rho_\delta)]\right\}\nonumber\\
&+\int^T_0\int_{D}\left\{-\psi\left(a\rho^\gamma_\delta+\delta\rho^\beta_\delta\right)\nabla\eta A[T_k(\rho_\delta)]+\psi[(\rho_\delta u_\delta\otimes u_\delta)\nabla\eta]A[T_k(\rho_\delta)]\right\}\nonumber\\
&+\int^T_0\!\!\int_{D}\!\left\{-\psi\eta(\nabla d_\delta\otimes\nabla d_\delta):(\nabla\Delta\nabla)[T_k(\rho_\delta)]-\psi[(\nabla d_\delta\otimes\nabla d_\delta)\nabla\eta]\mathcal{A}[T_k(\rho_\delta)]\right\}\nonumber\\
&+\int^T_0\!\!\int_{D}\!\left\{\frac{1}{2}\left(|\nabla d_\delta|^2+F(d_\delta)\right)\psi\eta T_k(\rho_\delta)+\frac{1}{2}\left(|\nabla d_\delta|^2+F(d_\delta)\right)\psi\nabla\eta \mathcal{A}[T_k(\rho_\delta)]\right\}\nonumber\\
&+\int^T_0\int_{D}\left\{\mu_3\psi\eta\left(N_\epsilon\otimes d_\epsilon-d_\epsilon\otimes N_\epsilon\right):(\nabla\Delta\nabla)[T_k(\rho_\delta)])-\psi_t\eta\rho_\delta u\mathcal{A}[T_k(\rho_\delta)]\right\}\nonumber\\
&+\int^T_0\int_{D}\left\{\psi\left((N_\epsilon\otimes d_\epsilon-d_\epsilon\otimes N_\epsilon)\nabla\eta\right)A[T_k(\rho_\delta)]+\psi\eta\rho_\delta u_\delta\mathcal{A}[div(T_k(\rho_\delta)u_\delta)]\right\}\nonumber\\
&+\int^T_0\int_{D}\psi\eta\rho_\delta u_\delta\mathcal{A}\left[\left(T_k'(\rho_\delta)\rho_\delta-T_k(\rho_\delta)\right)divu_\delta\right]\nonumber\\
&=\sum^{13}_{i=1}I^\delta_i.\nonumber
\end{align}
Using $\varphi=\psi\eta\mathcal{A}[\overline{T_k(\rho)}]$ as a test function of (\ref{tequ2}), we have
\begin{align}
&\int^T_0\int_{D}\psi\eta\left\{\overline{p}\overline{T_k(\rho)}-(\mu_4A+\mu_7tr(A)I):(\nabla\Delta^{-1}\nabla)[\overline{T_k(\rho)}]\right\}\nonumber\\
&=\int^T_0\int_{D}\left\{\psi(S\nabla\eta)A[\overline{T_k(\rho)}]-\psi\eta\rho u\otimes u:(\nabla\Delta\nabla)[\overline{T_k(\rho)}]\right\}\nonumber\\
&+\int^T_0\int_{D}\left\{-\psi\overline{p}\nabla\eta A[\overline{T_k(\rho)}]+\psi[(\rho u\otimes u)\nabla\eta]A[\overline{T_k(\rho)}]\right\}\nonumber\\
&+\int^T_0\int_{D}\left\{-\psi\eta(\nabla d\otimes\nabla d):(\nabla\Delta\nabla)[\overline{T_k(\rho)}]-\psi[(\nabla d\otimes\nabla d)\nabla\eta]\mathcal{A}[\overline{T_k(\rho)}]\right\}\nonumber
\end{align}
\begin{align}
&+\int^T_0\int_{D}\left\{\frac{1}{2}\left(|\nabla d|^2+F(d)\right)\psi\eta\overline{T_k(\rho)}+\frac{1}{2}\left(|\nabla d|^2+F(d)\right)\psi\nabla\eta \mathcal{A}[\overline{T_k(\rho)}]\right\}\nonumber\\
&+\int^T_0\int_{D}\left\{\psi\eta(N\otimes d-d\otimes N):(\nabla\Delta\nabla)[\overline{T_k(\rho)}])-\psi_t\eta\rho u\mathcal{A}[\overline{T_k(\rho)}]\right\}\nonumber\\
&+\int^T_0\int_{D}\left\{\psi((N\otimes d-d\otimes N)\nabla\eta)A[\rho]+\psi\eta\rho u\mathcal{A}[div(\overline{T_k(\rho)} u)]\right\}\nonumber\\
&+\int^T_0\int_{D}\psi\eta\rho u\mathcal{A}[\overline{(T_k'(\rho)\rho-T_k(\rho))divu}]\nonumber\\
&=\sum^{13}_{i=1}I_i.\nonumber
\end{align}
One easily observes
\begin{eqnarray}
\lim_{\delta\rightarrow0}\int^T_0\int_{D}\delta\rho^\beta_\delta T_k(\rho_\delta)\leq\lim_{\delta\rightarrow0}C\delta^{\frac{\beta}{\beta+\theta}}
\|\rho_\delta\|^\beta_{L^{\beta+\theta}((0,T)\times D)}=0,\nonumber
\end{eqnarray}
where $\theta$ is defined in {Lemma \ref{la2}}. Noting $T_k(z)$ is bounded,
we can get $I^\delta_i\rightarrow I_i$ term by term as in the proof of {Lemma \ref{lemma12}}.
\end{proof}
Next we define oscillations defect measure $OSC_p[\rho_\delta\rightarrow\rho](O)$, for $O\in((0,T)\times D)$(See \cite{f2}).
\begin{eqnarray}
OSC_p[\rho_\delta\rightarrow\rho](O)=\sup_{k\geq1}\overline{\lim_{\delta\rightarrow0}}\int_O|T_k(\rho_\delta)-T_k(\rho)|^p,
\end{eqnarray}
where $T_k$ is defined by (\ref{tk}).
\begin{lema}\label{la4}
Let $(\rho_\delta,u_\delta,d_\delta)$, $(\rho,u,d)$ satisfy the assumption of Lemma \ref{la3}. Then for any bounded set $O\subset((0,T)\times D)$, we have
\begin{eqnarray}
OSC_{\gamma+1}[\rho_\delta\rightarrow\rho](O)\leq C(O).\nonumber
\end{eqnarray}
\end{lema}
%%%%%%%%%%%%%%%%%%%%%%%%%%%%%%%%%%%%%%%%%%%%%%%%%%%%%%%%%%%%%%%%%%%%%%%%%%%%%%%%%%%%%%%%%%%%%%%%%%%%%%%%%%%%%%%%%%%%%%%%%%%%%%%%%%%%%%%%%%%
\begin{proof}
Using (\ref{b1}), we have
\begin{eqnarray}
\lim_{\delta\rightarrow0}\!\int^T_0\!\!\!\int_{D}\! p(\rho_\delta)T_k(\rho_\delta)-\overline{p(\rho)}\ \overline{T_k(\rho)}
=(\mu_4\!+\!\mu_7)\lim_{\delta\rightarrow0}\!\int^T_0\!\!\!\int_{D}\! divu_\delta T_k(\rho_\delta)-divu\overline{T_k(\rho)}.\nonumber
\end{eqnarray}
Moreover, one easily show,
\begin{align}
&\lim_{\delta\rightarrow0}\int^T_0\!\!\!\int_{D} \left[p(\rho_\delta)T_k(\rho_\delta)-\overline{p(\rho)}\ \overline{T_k(\rho)}\right]
=\lim_{\delta\rightarrow0}\int^T_0\!\!\!\int_{D}\left(p(\rho_\delta)\!-p(\rho)\right)\left(T_k(\rho_\delta)-T_k(\rho)\right)\nonumber\\
&+\!\int^T_0\!\!\!\int_{D}\left(\overline{p(\rho)}-p(\rho)\right)\!\!\!\left(T_k(\rho)-\overline{T_k(\rho)}\right)\geq\lim_{\delta\rightarrow0}
\int^T_0\!\!\!\int_{D}\left(p(\rho_\delta)\!-p(\rho)\right)\!\!\!\left(T_k(\rho_\delta)-T_k(\rho)\right).\label{osc1}
\end{align}
Here we use $p(z)$ is convex and $T_k(z)$ is concave. For $p(z)=az^\gamma$, we have
\begin{eqnarray}
p(y)-p(z)=\int^y_zp'(s)ds\geq\int^y_zp'(s-z)ds=p(y-z),\ y\geq z\geq0.\label{ss}
\end{eqnarray}
Considering the definition of $T_k$, the inequality (\ref{ss}) yields
\begin{eqnarray}
p(|T_k(y)-T_k(z)|)\leq p(|y-z|).
\end{eqnarray}
Thus we have
\begin{eqnarray}
&&a|T_k(\rho_\delta)-T_k(\rho)|^{\gamma+1}\leq p(|T_k(\rho_\delta)-T_k(\rho)|)\left|T_k(\rho_\delta)-T_k(\rho)\right|\nonumber\\
&&\leq p(|\rho_\delta-\rho|)|T_k(\rho_\delta)-T_k(\rho)|\leq (p(\rho_\delta)-p(\rho))(T_k(\rho_\delta)-T_k(\rho)).
\end{eqnarray}
Using (\ref{osc1}), we get
\begin{eqnarray}
\overline{\lim_{\delta\rightarrow0}}\int^T_0\int_{D}|T_k(\rho_\delta)-T_k(\rho)|^{\gamma+1}\leq\lim_{\delta\rightarrow0}\int^T_0\int_{D} p(\rho_\delta)T_k(\rho_\delta)-\overline{p(\rho)}\ \overline{T_k(\rho)}.\nonumber
\end{eqnarray}
And we easily obtain
\begin{align}
&\overline{\lim_{\delta\rightarrow0}}\int^T_0\int_{D}|T_k(\rho_\delta)-T_k(\rho)|^{\gamma+1}\nonumber\\
&\leq(\mu_4+\mu_7)\lim_{\delta\rightarrow0}\int^T_0\int_{D}\left[ divu_\delta T_k(\rho_\delta)-divu\overline{T_k(\rho)}\right]\nonumber\\
&=(\mu_4+\mu_7)\lim_{\delta\rightarrow0}\int^T_0\int_{D}\left(T_k(\rho_\delta)-T_k(\rho)+T_k(\rho)-\overline{T_k(\rho)}\right)divu_\delta \nonumber\\
&\leq2(\mu_4+\mu_7)\|divu\|_{L^2((0,T)\times D)}\overline{\lim_{\delta\rightarrow0}}\|T_k(\rho_\delta)-T_k(\rho)\|_{L^2((0,T)\times D)}\nonumber\\
&\leq C\overline{\lim_{\delta\rightarrow0}}\|T_k(\rho_\delta)-T_k(\rho)\|^2_{L^{\gamma+1}(D)}|D|^{\frac{\gamma-1}{2(\gamma+1)}}.
\end{align}
Then we complete the proof.
\end{proof}
\begin{lema}\label{la5}
Let $D\in C^{2+\alpha}$ be bounded in $\mathbb{R}^3$. $\left\{(\rho_\delta,u_\delta,d_\delta)\right\}$ is a sequence of solutions satisfying (\ref{sequ1})-(\ref{sequ3}) and $(\rho_\delta,u_\delta)$ is a renormalized solution of (\ref{sequ1}). Assume that $\rho_\delta\rightharpoonup^*\rho\ in\ L^\infty(0,T;L^\gamma(D))$ with $\gamma>\frac{2N}{N+2}$, $u_\delta\rightharpoonup u\ in\ L^2(0,T;H^1_0(D))$ and
$OSC_{\gamma+1}[\rho_\delta\rightarrow\rho](O)\leq C(O),\ O\subset((0,T)\times D)$. Then $(\rho,u)$ is a renormalized solution of (\ref{tequ1}).
\end{lema}
%%%%%%%%%%%%%%%%%%%%%%%%%%%%%%%%%%%%%%%%%%%%%%%%%%%%%%%%%%%%%%%%%%%%%%%%%%%%%%%%%%%%%%%%%%%%%%%%%%%%%%%%%%%%%%%%%%%%%%%%%%%%%%%%%%%%%%%%%%%%%
\begin{proof}
Using $(\rho_\delta,u_\delta)$ is renormalized solution of (\ref{sequ1}), we have
\begin{eqnarray}
\partial_tT_k(\rho_\delta)+div(T_k(\rho_\delta)u_\delta)+(T'_k(\rho_\delta)\rho_\delta-T_k(\rho_\delta))divu_\delta=0.
\end{eqnarray}
Passing to the limit for $\delta\rightarrow0$, we obtain
\begin{eqnarray}
\partial_t\overline{T_k(\rho)}+div(\overline{T_k(\rho)}u)+\overline{(T'_k(\rho)\rho-T_k(\rho))divu}=0,\label{la6}
\end{eqnarray}
Similarly to Lemma \ref{k1}, we get the following equation,
\begin{align}
\partial_tB(\overline{T_k(\rho)})+&div(B(\overline{T_k(\rho)})u)+B'(\overline{T_k(\rho)})\overline{T_k(\rho)}-B(\overline{T_k(\rho)})divu\nonumber\\
=&B'(\overline{T_k(\rho)})\overline{((T_k(\rho)-T'_k(\rho)\rho)divu)}\quad in\ D'((0,T)\times D),\label{la7}
\end{align}
where $B(z)$ satisfies
\begin{eqnarray}
B\in C^1[0,\infty),\quad B'(z)=0\ for\ all\ z\geq z_B.\nonumber
\end{eqnarray}
Utilizing the weak lower semi-continuity of function's norm, we deduce
\begin{eqnarray}
\|\overline{T_k(\rho)}-\rho\|_{L^1((0,T)\times D)}\leq\liminf_{\delta\rightarrow0}\|T_k(\rho_\delta)-\rho_\delta\|_{L^1((0,T)\times D)}\nonumber\\
\leq\sup_{\delta}\int_{\rho_\delta\geq k}\rho_\delta \leq k^{1-\gamma}\sup_{\delta}\|\rho_\delta\|_{L^\gamma((0,T)\times D)}.\nonumber
\end{eqnarray}
And passing to the limit for $k\rightarrow0$, we have
\begin{eqnarray}
B(\overline{T_k(\rho)})\rightarrow B(\rho),\ B'(\overline{T_k(\rho)})\rightarrow B'(\rho)\ \mathrm{in\ the\ corresponding\ spaces}.\nonumber
\end{eqnarray}
In order to complete the proof, we should show the right side of (\ref{la7}) tends to zero as $k\rightarrow\infty$. One easily show
\begin{align}
&\|B'(\overline{T_k(\rho)})\overline{((T_k(\rho)-T'_k(\rho)\rho)divu)}\|_{L^1((0,T)\times D)}\nonumber\\
&\leq\max_{z\geq0}|B'(z)|\int_{\overline{T_k(\rho)}\leq z_B}\left|\overline{((T_k(\rho)-T'_k(\rho)\rho)divu)}\right|\nonumber\\
&\leq\max_{z\geq0}|B'(z)|\sup_{\delta}\|divu_{\delta}\|_{L^2((0,T)\times D)}\liminf_{\delta\rightarrow0}\|T_k(\rho_\delta)-T'_k(\rho_\delta)\rho_\delta\|
_{L^2({\overline{T_k(\rho)}\leq z_B})}.\label{la8}
\end{align}
By interpolation, we get
\begin{align}
&\|T_k(\rho_\delta)-T'_k(\rho_\delta)\rho_\delta\|_{L^2({\overline{T_k(\rho)}\leq z_B})}\nonumber\\
&\leq\|T_k(\rho_\delta)-T'_k(\rho_\delta)\rho_\delta\|^{\frac{\gamma-1}{\gamma}}_{L^1((0,T)\times D)}\|T_k(\rho_\delta)-T'_k(\rho_\delta)\rho_\delta\|^{\frac{\gamma+1}{\gamma}}_{L^{\gamma+1}({\overline{T_k(\rho)}\leq z_B})}.\label{m1}
\end{align}
Observing the definition of $T_k$, we have
\begin{eqnarray}
\|T_k(\rho_\delta)-T'_k(\rho_\delta)\rho_\delta\|_{L^1((0,T)\times D)}\leq2^\gamma k^{1-\gamma}\sup_{\delta}\|\rho_\delta\|^\gamma_{L^\gamma((0,T)\times D)}.\label{m2}
\end{eqnarray}
Using {Lemma \ref{la4}} and $T'_k(z)z\leq T_k(z)$, we obtain
\begin{align}
&\overline{\lim_{\delta\rightarrow0}}\|T_k(\rho_\delta)-T'_k(\rho_\delta)\rho_\delta\|_{L^{\gamma+1}({\overline{T_k(\rho)}\leq z_B})}
\leq2\overline{\lim_{\delta\rightarrow0}}\|T_k(\rho_\delta)\|_{L^{\gamma+1}({\overline{T_k(\rho)}\leq z_B})}\nonumber\\
&\leq2\overline{\lim_{\delta\rightarrow0}}\|T_k(\rho_\delta)-T_k(\rho)\|_{L^{\gamma+1}((0,T)\times D)}
+2\|T_k(\rho)-\overline{T_k(\rho)}\|_{L^{\gamma+1}((0,T)\times D)}\nonumber\\
&\qquad+2\|\overline{T_k(\rho)}\|_{L^{\gamma+1}(\{\overline{T_k(\rho)}\leq z_B\})}\nonumber\\
&\leq4OSC[\rho_\delta\rightarrow\rho]_{\gamma+1}((0,T)\times D)+2z_B(T|\Omega|)^\frac{1}{\gamma+1}.\label{aa1}
\end{align}
 Substituting (\ref{m1})-(\ref{aa1}) into (\ref{la8}), and passing to the limit for $k\rightarrow\infty$, (\ref{la8}) tends to be zero. For general $B$ in Definition \ref{defi}, we can use $z_B=k^{\frac{(\gamma-1)^2}{2(\gamma+1)}}$. Then (\ref{aa1}) still holds.
\end{proof}

%%%%%%%%%%%%%%%%%%%%%%%%%%%%%%%%%%%%%%%%%%%%%%%%%%%%%%%%%%%%%%%%%%%%%%%%%%%%%%%%%%%%%%%%%%%%%%%%%%%%%%%%%%%%%%%%%%%%%%%%%%%%%%%%%%%%%%%%%%%%%
In the last step, we prove $\rho_\delta\rightarrow\rho\ in\ L^1((0,T)\times D)$. Let's define $L_k$:
\begin{eqnarray}
L_k(z)=
\left\{\begin{array}{ll}
z\log(z)\quad &0\leq z\leq 1, \\
z\int^z_1\frac{T_k(s)}{s^2}\ &z>1.
\end{array}\right.\nonumber
\end{eqnarray}
Observing $(\rho_\delta,u_\delta)$ is a renormalized solution of (\ref{sequ1}) and $(\rho,u)$ is a renormalized solution of (\ref{tequ1}), we have
\begin{eqnarray}
\partial_tL_k(\rho_\delta)+div(L_k(\rho_\delta)u_\delta)+T_k(\rho_\delta)divu_\delta=0,\label{aa2}
\end{eqnarray}
\begin{eqnarray}
\partial_tL_k(\rho)+div(L_k(\rho)u)+T_k(\rho)divu=0.\label{aa3}
\end{eqnarray}
Using (\ref{aa2}), we get
\begin{align}
&L_k(\rho_\delta)\rightarrow \overline{L_k(\rho)}\quad in\ C([0,T];L^\gamma_{weak}(D)),\\
&\rho_\delta\log(\rho_\delta)\rightarrow\overline{\rho\log(\rho)}\quad in\ C([0,T];L^\alpha_{weak}(D)),\ 1\leq\alpha<\gamma.
\end{align}
Let $\delta\rightarrow0$ in (\ref{aa2}),
\begin{eqnarray}
\partial_t\overline{L_k(\rho)}+div(\overline{L_k(\rho)}u)+\overline{T_k(\rho)divu}=0.\label{aa4}
\end{eqnarray}
(\ref{aa3}) and (\ref{aa4}) yield
\begin{eqnarray}
\partial_t(\overline{L_k(\rho)}-L_k(\rho))+div(\overline{L_k(\rho)}u-L_k(\rho)u)+\overline{T_k(\rho)divu}-T_k(\rho)divu=0.\label{aa5}
\end{eqnarray}
For any $\eta_n\in C^\infty_0(D)$, It holds
\begin{align}
\lim_{\delta\rightarrow0}\int_{D}\eta_n\left[L_k(\rho_{0,\delta})-L_k(\rho_0)\right]dx=0.
\end{align}
Then we have
\begin{align}
&\int_{D}\eta_n\left[\overline{L_k(\rho)}-L_k(\rho)\right](t)dx\nonumber\\
&=\int^T_0\int_{D}\left[\overline{L_k(\rho)}u-L_k(\rho)u)\nabla\eta_n+(T_k(\rho)divu-\overline{T_k(\rho)divu}\right]\eta_n.\label{eta}
\end{align}
From the definition of $L_k$, we can get $L_k(z)=2kz-2k,z\geq 3k$. Let $\{\eta_n\}$ be sequence such that $\eta_n\rightarrow1\ in\ D$. Passing to the limit for $n\rightarrow\infty$ in (\ref{eta}), we have
\begin{align}
\int_{D}\left[\overline{L_k(\rho)}-L_k(\rho)\right](t)dx
=\int^T_0\int_{D}\left[ T_k(\rho)divu-\overline{T_k(\rho)divu}\right].\nonumber
\end{align}
Also it holds
\begin{align}
&\int^T_0\int_{D}\left[ T_k(\rho)divu-\overline{T_k(\rho)divu}\right]\nonumber\\
&=\int^T_0\int_{D} \left[T_k(\rho)divu-\overline{T_k(\rho)}divu+\overline{T_k(\rho)}divu-\overline{T_k(\rho)divu}\right]\nonumber\\
&\leq\lim_{\delta\rightarrow0}\int^T_0\int_{D}\left[ divu\overline{T_k(\rho)}-T_k(\rho_\delta)divu_\delta\right]\nonumber\\&+\|divu\|_{L^2((0,T)\times D)}\|\overline{T_k(\rho)}-T_k(\rho)\|_{L^2((0,T)\times D)}.\label{resl}
\end{align}
Using Lemma \ref{la3} and the weak lower semi-continuity of the function's norm, one obtains
\begin{align}
&\lim_{\delta\rightarrow0}\int^T_0\int_{D}\left[ divu\overline{T_k(\rho)}-T_k(\rho_\delta)divu_\delta\right]\nonumber\\
&=\frac{a}{\mu_4+\mu_7}\lim_{\delta\rightarrow0}\int^T_0\int_{D}\left[\overline{\rho^\gamma}\overline{T_k(\rho)}-\rho^\gamma_\delta T_k(\rho_k)\right]\leq0.
\end{align}
For the second term of (\ref{resl}), we have
\begin{align}
&\|\overline{T_k(\rho)}-T_k(\rho)\|_{L^2((0,T)\times D)}\nonumber\\
&\leq\|\overline{T_k(\rho)}-T_k(\rho)\|^{\frac{\gamma-1}{2\gamma}}_{L^1((0,T)\times D)}\|\overline{T_k(\rho)}-T_k(\rho)\|^{\frac{\gamma+1}{2\gamma}}_{L^{\gamma+1}((0,T)\times D)}.\nonumber
\end{align}
Utilizing Lemma \ref{la4}, we have
\begin{eqnarray}
\|\overline{T_k(\rho)}-T_k(\rho)\|_{L^{\gamma+1}((0,T)\times D)}\leq C.
\end{eqnarray}
And one easily gets
\begin{align}
&\|\overline{T_k(\rho)}-T_k(\rho)\|_{L^1((0,T)\times D)}\nonumber\\
&\leq\|\overline{T_k(\rho)}-\rho\|_{L^1((0,T)\times D)}+\|T_k(\rho)-\rho\|_{L^1((0,T)\times D)}\nonumber\\
&\leq\lim_{\delta\rightarrow0}\|T_k(\rho_\delta)-\rho_\delta\|_{L^1((0,T)\times D)}+\|T_k(\rho)-\rho\|_{L^1((0,T)\times D)}\nonumber\\
&\leq k^{1-\gamma}\sup_{\delta}\|\rho_\delta\|_{L^\gamma((0,T)\times D)}+k^{1-\gamma}\|\rho\|_{L^\gamma((0,T)\times D)}
\rightarrow0\quad\ as\ k\rightarrow\infty.\nonumber
\end{align}
Then we have
\begin{eqnarray}
\overline{\rho\log(\rho)}(t)\leq\rho\log(\rho)(t).\nonumber
\end{eqnarray}
The convex property of $\rho\log(\rho)$ yields
\begin{eqnarray}
\overline{\rho\log(\rho)}(t)\geq\rho\log(\rho)(t).\nonumber
\end{eqnarray}
Then we get the following result:
\begin{eqnarray}
\overline{\rho\log(\rho)}(t)=\rho\log(\rho)(t).\nonumber
\end{eqnarray}
which leads to $\rho_\delta\rightarrow\rho\ in\ L^1((0,T)\times D)$. Thus we have proved $\overline{\rho^\gamma}=\rho^\gamma$. The proof of Theorem \ref{th1} is completed.

\section{The proof of Theorem \ref{th2}}\label{5}

\setcounter{equation}{0}

First we consider the system (\ref{equa1})-(\ref{equa3}) in a bounded domain $B_r$(a bounded ball with radius r and center on the origin). Using the cut off functions and mollified operator, the initial conditions are constructed as follows:
\begin{eqnarray}
\rho_{0,r}=\left.\rho_0\right|_{B_r},\quad d_{0.r}=\left.d_0\right|_{B_r}.\nonumber
\end{eqnarray}
such that
\begin{eqnarray}
\left\{\!\!\!\!\!\!\begin{array}{ll}
&\rho_r(x,0)=\rho_{0,r}(x)\geq 0\quad a.e.\ in\ B_r, \\
&(\rho_r u_r)(x,0)=q_{0,r}(x),q_{0,r}(x)=0\ a.e.\ on\ \{\rho_{0,r}(x)=0\},\frac{|q_{0,r}|^2}{\rho_{0,r}}\in L^1(B_r),\\
&d_r(x,0)=d_{0,r}(x), \quad |d_{0,r}(x)|=1,\ d_{0,r}\in H^2(B_r),\\
&u_r(x,t)=0,\quad d_r(x,t)=d_{0,r}(x),\quad \quad(x,t)\in\partial B_r\times(0,\infty),
\end{array}\right.\label{5.1}
\end{eqnarray}
and
\begin{eqnarray}
\int_{\mathbb{R}^3}(\rho_{0,r})^\gamma-\gamma(\rho_{0,r}-1)-1\leq \tilde{C}_0.\label{5.2}
\end{eqnarray}
%The boundary conditions are
%\begin{eqnarray}
%u_r(x,t)=0,\quad d_r(x,t)=d_{0,r}(x),\quad \quad(x,t)\in\partial D\times(0,\infty).\label{5.3}
%\end{eqnarray}
There exists a constant $E_0$, such that
\begin{eqnarray}
E_{0,r}=\int_{D}\left[\frac{|q_{0,r}|^2}{\rho_{0,r}}+\frac{1}{2}|\nabla d_{0,r}|^2+\frac{(\rho_{0,r})^\gamma-\gamma(\rho_{0,r}-1)-1}{\gamma-1}+F(d_{0,r})\right]\leq E_0.\nonumber
\end{eqnarray}
The existence of weak solution to system (\ref{equa1})-(\ref{equa3}) with (\ref{5.1})-(\ref{5.2}) can be guaranteed by Theorem \ref{th1}.
Using the energy inequality, we have
\begin{align}
&\left\|\sqrt{\rho_r}u_r\right\|_{L^\infty(0,T;L^2(B_r))}\leq E_0,\quad \left\|\nabla u_r\right\|_{L^2(0,T\times B_r)}\leq E_0,\label{5.4}\\
&\left\|N_r\right\|_{L^2(0,T\times B_r)}\leq E_0,\quad \left\|\nabla d_r\right\|_{L^\infty(0,T;L^2(B_r))}\leq E_0,\label{5.5}\\
&\int_{\mathbb{R}^3}(\rho_{r})^\gamma-\gamma(\rho_{r}-1)-1\leq E_0.\label{5.6}
\end{align}
Using (\ref{5.6}) and the following fact:
\begin{align}
\left\{\begin{array}{ll}
&x^\gamma-1-\gamma(x-1)\geq\nu|x-1|^\gamma,\quad \gamma\geq2,\\
&x^\gamma-1-\gamma(x-1)\geq\nu|x-1|^2,\quad \gamma<2,\ |x-1|\leq \frac{1}{2},\\
&x^\gamma-1-\gamma(x-1)\geq\nu|x-1|^\gamma,\quad \gamma<2,\ |x-1|\geq \frac{1}{2},
\end{array}\right.\label{5.7}
\end{align}
we have
\begin{align}
\left\{\begin{array}{ll}
&\left\|\rho_r-1\right\|_{L^\infty(0,T;L^\gamma(B_r))}\leq E_0,\ if\ \gamma\geq2,\\
&\left\|\rho_r-1\right\|_{L^\infty(0,T;L_2^\gamma(B_r))}\leq E_0,\ if\ \gamma\leq2.
\end{array}\right.\label{5.10}
\end{align}
We split $u_r$ as follows:
\begin{align}
&u_r=u_r^1+u_r^2,\nonumber\\
&u_r^1=u_r|_{|\rho_r-1|\leq\frac{1}{2}},\quad u_r^2=u_r|_{|\rho_\epsilon-1|\geq\frac{1}{2}}.\nonumber
\end{align}
Then we have
\begin{align}
&\sup_t\int_{B_r}|u_r^1|^2dx\leq2\sup_t\int_{B_r}\rho_r|u_r|^2dx\leq E_0\nonumber\\
&\int_{B_r}|u_r^2|^2dx\leq2\int_{B_r}|\rho_r-1||u_r|^2dx\nonumber\\
&\leq\|\rho_r-1\|_{L^\gamma(B_r)}\|u_r^2\|_{L^2(B_r)}^\theta\|\nabla u_r\|_{L^2(B_r)}^{1-\theta},\nonumber
\end{align}
which is
\begin{align}
\|u_r^1\|_{L^\infty((0,T);L^2(B_r))}\leq E_0,\label{5.8}\\
\|u_r^2\|_{L^2((0,T)\times B_r)}\leq C(E_0).\label{5.9}
\end{align}
Thus we have deduced $u_r\in L^2(0,T;H_0^1(B_r))$ whose bound is independent of $r$. Like in the bounded case, we have $|d_r|\leq1\ a.e.\ in\ B_r$ independing of $r$. Noticing
\begin{align}
\|f(d)\|_{L^\infty(0,T);L^2(B_r)}\leq\|F(d)\|_{L^\infty(0,T);L^2(B_r)}\leq C(E_0),\nonumber
\end{align}
and the bound of $N_r$ in (\ref{5.5}), we have
\begin{align}
 \|\nabla^2d_r\|_{L^2((0,T)\times B_r)}\leq C(E_0),\ \mathrm{independent\ of}\ r,\nonumber
\end{align}
by using interior elliptic estimate on $\Delta d_r-f(d_r)=\frac{1}{\lambda_1}N_r$.
Observing $N_r=d_{r t}+(u_r\cdot\nabla)d_r-\Omega_r d_r$, we have
\begin{align}
\left\|d_{r t}\right\|_{L^2((0,T);L^{\frac{3}{2}}(B_r))}\leq C(E_0),\ \ \mathrm{independent\ of}\ r.\nonumber
\end{align}
Let us prolong $(\rho_r,u_r,d_r)$ by zero outside $B_r$. And for the sake of convenience, we still use $(\rho_r,u_r,d_r)$.
It exists $(\rho,u,d)$ such that we have the following:
\begin{align}
&\rho_r-\rho\rightharpoonup^*0\qquad in\ L^\infty(0,T;L^\gamma(\mathbb{R}^3)),\ if\ \gamma\geq2,\nonumber\\
&\rho_r-\rho\rightharpoonup^*0\qquad in\ L^\infty(0,T;L_2^\gamma(\mathbb{R}^3)),\ if\ \gamma<2,\nonumber\\
&u_r\rightharpoonup u\qquad in\ L^2(0,T;H^1(\mathbb{R}^3)),\nonumber\\
&d_r\rightharpoonup^* d\qquad in\ L^\infty(0,T;\mathcal{H}(\mathbb{R}^3)),\nonumber\\
&d_r\rightharpoonup d\qquad in\ L^2(0,T;\mathcal{H}(\mathbb{R}^3)\cap \dot{H}^2(\mathbb{R}^3)).\nonumber
\end{align}
At last, we only need to show $(\rho,u,d)$ satisfies (\ref{equa1})-(\ref{equa3}) in $D_{loc}'((0,T)\times \mathbb{R}^3)$. Observing that it's nothing but the case of bounded domain(at least, we can use the same process). So we end the prove of Theorem \ref{th2}.

% You may incorporate your references as follows in your main tex file.
% Using BibTex is not recommended but can be handled.

\medskip
% The data information below will be filled by AIMS editorial staff
Received January 2012; revised November 2012.
\medskip

\end{document}